\documentclass[10pt]{amsart}

\usepackage[utf8x]{inputenc}

\usepackage{dsfont}
\usepackage{amsmath}
\usepackage{amssymb}
\usepackage{graphicx}
\usepackage{amssymb}
\usepackage{amsfonts}
\usepackage{soul}
\usepackage{mathpazo}
\usepackage{color}
\usepackage{yfonts}
\usepackage{paralist}
\usepackage{stmaryrd}
\usepackage{amsxtra}
\usepackage{latexsym,mathrsfs}
\usepackage{enumerate}
\usepackage{mathabx}

\usepackage{epsfig, enumerate}
\usepackage{color}
\usepackage{hyperref}
\newtheorem{theorem}{Theorem}
\newtheorem{lemma}[theorem]{Lemma}
\newtheorem{Remark}{Remark}

\newtheorem{coro}[theorem]{Corollary}
\newtheorem{definition}{Definition}
\newtheorem{prop}[theorem]{Proposition}
\newtheorem{question}{Question}

\numberwithin{theorem}{section}
\numberwithin{Remark}{section}
\numberwithin{Example}{section}
\numberwithin{equation}{section}

\title[Stable and Unstable Behaviour]{On Stable and Unstable Behaviour of Certain Rotation Segments}
\author{Salvador Addas-Zanata}
\address[Addas-Zanata]{Instituto de Matem\' atica e Estat\' istica da Universidade de S\~ao Paulo,
R. do Mat\~ ao, 1010 - Cidade Universitaria, S\~ ao Paulo, Brasil} 
\email{sazanata@ime.usp.br}

\author{Xiao-Chuan Liu}
\address[Liu]{Instituto de Matem\' atica e Estat\' istica da Universidade de S\~ao Paulo,
R. do Mat\~ ao, 1010 - Cidade Universitaria, S\~ ao Paulo, Brasil} 
\email{lxc1984@gmail.com}
\begin{document}

\begin{abstract}
In this paper, we study 
 non-wandering homeomorphisms of the two torus homotopic to
 the identity, 
 whose rotation sets
 are non-trivial segments from $(0,0)$ to 
some totally irrational point $(\alpha,\beta)$. 
We show that for any $r \geq 1,$ this rotation set only appears for 
$C^r$ diffeomorphisms satisfying some degenerate conditions.  
And when such a rotation set does appear, 
assuming several natural conditions 
that are generically satisfied in the area-preserving world,
we give a clearer 
description of its 
rotational behaviour. 
More precisely, 
the dynamics admits bounded deviation along the direction 
$-(\alpha,\beta)$ in the lift, and the rotation set is locked inside
an arbitrarily small 
cone with respect to
 small $C^0$-perturbations of the dynamics. 
On the other hand, 
for any non-wandering
homeomorphism $f$ with this kind of rotation set, 
we also present a 
perturbation scheme in order for the rotation set to be eaten 
by rotation set of some nearby dynamics, in the sense that 
the later set has non-empty interior and contains the former one. 
These two flavors interplay and share the common goal 
of understanding the stability/instability properties 
of this kind of rotation set. 
\end{abstract}

\maketitle{}

\tableofcontents

\section{Introduction}

The notion of a rotation number was introduced by Poincar\'e 
in order to gather information on the average "rotational" linear speed of a 
dynamical system. 
Rotation theory is well understood 
 only for circle homeomorphisms 
 or endomorphisms, 
  and it
 is still a great source of 
 problems in two dimensional manifolds (annulus, two torus, higher genus 
 surfaces, and so on).
 Focusing only on the two torus, 
 there
 are two most important and related topics, namely, 
 the shape of the rotation set 
 and how it changes depending on the dynamics. 
 In this paper, we will work on both topics, based on one specific 
 type of rotation set. 
 On the one hand, we try to understand the 
variation of a rotation set depending on the 
homeomorphism, under different regularities. On the other hand, 
we study how a certain shape of the rotation set restricts the 
dynamics.

In general, 
with respect to the Hausdorff topology,
 the rotation set varies 
 upper-semicontinuously.
 It is also known that if 
 the rotation set of some $f$ has non-empty interior, 
 then it is in fact continuous at  
 $f$ (see ~\cite{MZ} and ~\cite{MZ2}).  
 Moreover, 
 for a $C^0$-open and dense subset of
homeomorphisms, the rotation set is stable (i.e., it remains unchanged 
under small perturbations, 
see~\cite{Passeggi_rational} and \cite{Andres_C0}). Note that this is true 
both in the 
set of all homeomorphisms, and in the set of area-preserving ones. 

Below, in order to state our main results,
 we will use some notations that are mostly standard, and 
postpone their precise definitions to Section~\ref{preliminary}.

Our objective is to look 
at an interesting situation, 
where the rotation set is a segment connecting $(0,0)$ to some totally irrational 
point $(\alpha,\beta)$. 
We want to understand 
how it can be changed under sufficiently small
$C^0$ perturbations, and 
what properties should such a  homeomorphism satisfy. 
An Example with such a rotation set was first described 
in~\cite{Handel_no_periodic_orbit} by Handel, who attributed it to Katok. 
It is produced by a slowing-down 
procedure from a constant speed irrational flow.
 A smooth area-preserving example was obtained 
by Addas-Zanata and Tal 
in~\cite{Salvador_fareast}. See also the paper~\cite{Kwapisz_Diophantine} 
by Kwapisz and Mathison which shows
some special ergodic properties 
for an explicit example.  

In the $C^0$ category, the more precise task for us is
to study when and how 
the rotation set of a dynamical system is 
ready to grow.   
%This kind of projects were 
%initiated with 
One of the first results on this subject appeared in ~\cite{Salvador}, where it was proved that if the rotation set $\rho(\widetilde f)$ 
 contains a non-rational vector $(\alpha,\beta)$ as an extremal point, 
 then for any supporting line $r$ (i.e., a straight line containing $(\alpha, \beta)$, 
 such that the rotation set avoids one
connected component of its complement, denoted $O$), 
 by certain arbitrarily small $C^0$-perturbations, 
 the new rotation set will intersect $O$
 (See also \cite{Guiheneuf_instability} for a $C^1$ version of this theorem). 
  In this paper, we are able to go one step further in this direction.

\begin{theorem}\label{instability_general}
Let $\widetilde f \in \widetilde{\text{Homeo}}_{0,\text{nw}}(\Bbb T^2)$, whose 
rotation set $\rho(\widetilde f)$ is a segment from $(0,0)$
to some totally irrational point $(\alpha,\beta)$.
For any $\varepsilon>0$, there exists 
$\widetilde g \in \widetilde {\text{Homeo}}_0(\Bbb T^2)$
 with $\text{dist}_{C^0}(\widetilde g, \widetilde f)<\varepsilon$, 
such that,
 $\rho(\widetilde g)$ has non-empty interior, 
 and $\rho(\widetilde f)\backslash \{(0,0)\}
  \subset \text{Int}(\rho(\widetilde g))$. 
\end{theorem}

\begin{Remark} 
If the map $\widetilde{f}$ from the above theorem preserves area, then the perturbed map
$\widetilde{g}$ can also be chosen in the area-preserving world. This is a 
well-known fact (follows from Lemma \ref{firstclosing}), but it deserves to be mentioned, as 
it is one of the only cases known 
to us on how to perturb a non-wandering homeomorphism and remain non-wandering.
\end{Remark}

It is interesting to ask if the same statement is 
also true in $C^1$ topology. 
If we go on to consider $C^r$ diffeomorphisms,
$r\geq 1$, 
clearer descriptions
should be expected. 
In particular, we proved the following result, which suggests
 the non-genericity of the set of non-wandering diffeomorphisms 
 with this special kind of rotation set. We say "suggest", because 
it is not known how to perturb a non-wandering diffeomorphisms and remain 
non-wandering (unless of course, in some particular cases, like area-preserving maps). 

\begin{theorem}\label{generic_case} 
Let 
$\widetilde f\in 
\widetilde{\text{Diff}}_{0,\text{nw}} (\Bbb T^2)$. 
Assume that for every integer $n>0$, the linear part $Df^n$ computed 
at each $n$-periodic point does not have $1$ as an eigenvalue and 
there are no saddle-connections. 
Then, the rotation set  $\rho(\widetilde f)$ 
is not  a segment from $(0,0)$ to some totally irrational point 
$(\alpha,\beta)$.
\end{theorem} 

Note that the above conditions are satisfied by 
generic area-preserving $C^r$ diffeomorphisms, for all $r \geq 1.$
Naturally, the next
 task is the following. 
How can we 
understand a 
typical non-wandering 
diffeomorphism $f$ 
which does admit such a rotation set?
This seems to be hard, 
because very little is known on the set of non-wandering homeomorphisms
or diffeomorphisms of a surface. 
As we said above, 
unlike in the area-preserving case,
there does not exist a method 
available to make a perturbation 
within the set of 
non-wandering homeomorphisms. 

Nevertheless, in the area-preserving setting, 
if one wants to obtain more 
information on
 the non-generic diffeomorphisms, traditionally,
 one works with generic families. 
This inspires us to 
formulate nice 
conditions, which hold true 
in a broader set 
of diffeomorphisms. This approach 
eventually helps us to detect properties, 
which general non-wandering homeomorphisms might
satisfy. 
Along this direction, 
we obtain 
the following two results.

\begin{theorem}\label{carac_map}
Suppose $f\in \text{Diff}_{0,\text{nw}}(\Bbb T^2)$ 
satisfies certain natural conditions, 
% which are all  satisfied by any 
% $C^r$-generic family of
%area-preserving diffeomorphisms 
 see Definition~(\ref{generic_in_generic}).
 Suppose also that $\rho(\widetilde {f})$ for some lift $\widetilde f$ 
 is a segment from $(0,0)$
to a totally irrational point $(\alpha,\beta).$ 
Then, $f$ has (finitely many) fixed points 
(and no periodic point which is not fixed), all with $0$ 
topological index, and the local dynamics around them is obtained by gluing 
exactly two hyperbolic sectors. The stable branch 
of any of these fixed points does not intersect the unstable branch of any other point. 
Moreover, in the plane, for each fixed point, its unstable and stable 
branches are bounded in the direction orthogonal to $(\alpha,\beta)$ and 
the unstable (resp. stable) branch goes to infinity following the vector $(\alpha,\beta)$
(resp. $-(\alpha,\beta)$). 
Finally, any stable or unstable branch of a fixed point is dense in the torus.
\end{theorem}

\begin{theorem}\label{generic_family_not_perturbable}
Under the same hypotheses of the above Theorem,
%Suppose $f\in \text{Diff}_{0,\text{nw}}(\Bbb T^2)$ 
%satisfies certain natural conditions, 
% which are all  satisfied by any 
% $C^r$-generic family of
%area-preserving diffeomorphisms 
 %see Definition~(\ref{generic_in_generic}).
 %Suppose also that $\rho(\widetilde {f})$ for some lift $\widetilde f$ 
 %is a segment from $(0,0)$
%to a totally irrational point $(\alpha,\beta).$ 
for any non-zero integer vector $(a,b),$ $a$ and $b$ coprimes, there exists a simple
closed curve $\theta$ in $\Bbb T^2,$ with homological direction $(a,b),$ 
such that any connected component of the lift of $\theta$ to the plane is a Brouwer line for $\widetilde{f}$.  
Moreover, for any straight line $\gamma$ containing 
$(0,0)$ and avoiding $(\alpha,\beta)$,
there exists $\varepsilon>0$, such that, for any 
$\widetilde g\in \widetilde{\text{Homeo}}_{0}(\Bbb T^2)$, with 
$d_{C^0}(\widetilde f,\widetilde g)<\varepsilon$,
$\rho(\widetilde g)$ is contained in the union of $\gamma$ 
and one of the connected components of its complement, the one which contains $\rho(\widetilde f)\backslash \{(0,0)\}$.

\end{theorem}

The next Corollary is a direct consequence of Theorem \ref{generic_family_not_perturbable}. 
Nevertheless, we will 
actually prove the Corollary before the theorem (see Lemma~\ref{Generic_Bounded_in_(-alpha,-beta)}). 

\begin{coro}\label{coro_bounded} Let $f$ satisfy  the conditions in
 Theorem~\ref{generic_family_not_perturbable}. 
Then it has bounded deviation along the direction 
$-(\alpha,\beta)$.
\end{coro}

We also consider the unbounded deviation along the direction $(\alpha,\beta)$.
The next theorem requires one more condition, the existence of an invariant foliation, which 
is clearly satisfied in the particular case when
$f$ is the time-one map of some flow. 
See its more precise statement in Section 7. 

\begin{theorem}\label{generic_bounded_unbounded}
Consider any $\widetilde f$ as in Theorem~\ref{generic_family_not_perturbable}.
If $f$ preserves a $C^0$ foliation of $\Bbb T^2$, then 
$\widetilde f$ has unbounded deviation along $(\alpha,\beta)$.
\end{theorem}

This last result leads to the following 
interesting
question. 
\begin{question}\label{Question_unbounded}
	 For $\widetilde f$ which lifts $f\in \text{Homeo}_{0,\text{nw}}(\Bbb T^2)$, 
  suppose $\rho(\widetilde f)$ is the line segment from $(0,0)$ to a totally
  irrational $(\alpha,\beta)$.
Is it true that $\widetilde f$ has
 unbounded deviation along the direction $(\alpha,\beta)$?
\end{question}

To conclude, let us briefly describe the organisation of this paper and the main
ideas used in the proofs. 

In Section 2, we will summarise some notations and previous results that will be used along 
the text. 

In Section 3, we introduce a perturbation technique, 
which is very useful under the condition of unbounded deviations along some direction.
The purpose is to find 
$\varepsilon$-pseudo periodic orbits, 
 which can be "closed" in order to become periodic orbits, so with rational rotation vector. 
 The difficulty is that, 
 a priory, the original method in \cite{Salvador} does not give enough information 
 to locate the position of these rational numbers, except that they are outside the rotation set
 $\rho(\widetilde f)$. 
 
 In Section 4, we focus on the case when the map has bounded deviations. 
 We apply a result proved by J{\"a}ger 
(see~\cite{Linearization}) in order to obtain 
 a semi-conjugacy between the restriction of the dynamics to a certain 
 minimal set and the rigid torus rotation. 
 Then, we prove that whenever one can perturb the rigid rotation, we
 can also perturb the original homeomorphism. Combining both results from Section 3 and 
 Section 4, we complete the proof of Theorem~\ref{instability_general}. 
 
 From Section 5, we start working with diffeomorphisms. 
 There, we prove Theorem~\ref{generic_case}.
 There are three main ingredients in this proof. The first one belongs to the theory of
  prime ends rotation numbers. 
  The second one is the so called \textit{bounded disk lemma},
  firstly proved by Koropecki and Tal in \cite{Strictly_Toral}. 
  The third one consists of certain properties of invariant branches at hyperbolic periodic saddles, mostly from Fernando Oliveira's paper \cite{Oliveira}.

%
%
%CHECK IF REFERENCES BELOW {Dumortier} and {Melo} ARE CORRECT. THEY 
%SHOULD BE: 
%
%
%
%\bibitem{Dumortier}  Dumortier F., Rodrigues P. and Roussarie R.
%(1981):\ Germs of Diffeomorphisms in the Plane. {\it Lecture Notes in %Mathematics }{\bf 902}, Springer-Verlag, New York-Berlin.
%
%
%\bibitem{Melo}  Dumortier F. and Melo W. (1988): A Type of Moduli
%for Saddle Connections of Planar Diffeomorphisms. {\it J. Diff. Eq.}
%{\bf %75}, 88-102.
%
%

 In Section 6, we work with diffeomorphisms satisfying   
 certain conditions,%. These conditions are much weaker, 
 which in the area-preserving case, are generic in the complement of
 the set of maps which satisfy the hypotheses of theorem \ref{generic_case}. 
See for instance \cite{Dumortier} and \cite{Melo}. 
First, we collect several results describing dynamical properties of maps which
satisfy the conditions in Definition~(\ref{generic_in_generic}). These results 
together imply Theorem \ref{carac_map} and are also an important part
 in the proof of existence of the Brouwer lines (Theorem 
\ref{generic_family_not_perturbable}).

    In Section 7, 
    we continue to study diffeomorphisms satisfying the
    conditions from Section 6 and prove 
    Theorem~\ref{generic_bounded_unbounded}.
  \\
  
 \noindent \textbf{Notational Remark.} There will be 
 a small abuse of notation among the text below. 
 For example,
 when we introduce integers, positive constants along 
 the arguments,  
 choices will be made differently in different 
 subsections, sometimes with the same name. 
 However, they will be consistent 
 within one single subsection. \\
   
\noindent \textbf{Acknowledgements.} We thank Andres Koropecki and 
F\' abio Armando Tal for many helpful discussions. %Part of the work was done when 
% both authors were participating in the
 %Workshop on Surface Dynamics, held 
 %in Montevideo, Uruguay, December 17-20, 2018. We thank the organizers for 
 %their hospitality. 
S.A-Z. was partially supported by CNPq grant 
(Grant number 306348/2015-2).
X-C.L. is supported by Fapesp P\'os-Doutorado grant
 (Grant Number 2018/03762-2).

\section{Preliminaries and Previous Results}\label{preliminary}
The main purpose of this section is to fix notations and to
recall some previous results for later use. In some cases, the 
formulation contains some minor variations from the reference, 
and we will give some 
short proofs only stressing the differences. 
We will also show some elementary lemmas as well. 

Note that some of the notations were already used in the statements of the
theorems in the introduction. 
\subsection{Planar Topology and Dynamics}\label{notation1}\hfill\\

For any planar subset $M$, denote by $\text{Inter}(M)$ the interior of 
$M$, and by $\partial M$ the boundary of $M$.
The following property in planar topology will be used. We say a planar set $F$ separates
the points $x$ and $y$ if they are in different connected components of $F^c.$

\begin{lemma}\label{Newman}[Theorem 14.3 in Chapter V of~\cite{Newman_elements}]. 
Let $F$ be a closed subset of the plane $\Bbb R^2$,
separating two points $x$ and $y$. Then some connected component of $F$
also separates $x$ and $y$.
\end{lemma}

% ADJUST THE REMARK BELOW TO THE TEXT PATERN, PLEASE.

%Remark: Separating the points $x$ and $y$ means that they are in different connected components of $F^c.$

%Let $\Gamma$ be the image of some injective 
%continuous map $\beta:\Bbb R\to S^2$. Suppose $f:S^2\to S^2$ is a homeomorphism.
%We call $\Gamma$ a
%\textit{translation line} if $f(\Gamma)=\Gamma$, and 
%the $\beta^{-1}\circ f\big|_{\Gamma} \circ \beta$ is conjugated to 
%the rigid translation $x\mapsto x+1$. 
%Denote the omega limit set
%$\omega(\Gamma)=\bigcap_{t\to \infty} \overline {\beta([t,+\infty)}$. Then we need the following main 
%result from \cite{poincare_andres_rata}

%\begin{lemma}[Theorem A in \cite{poincare_andres_rata}]\label{poincare_bendixson}
%Suppose $f$ is non-wandering in some neighbourhood of $\Gamma$. 
%Then, the omega limit set $\omega(\Gamma)$ contains at least one $f$-fixed point.
%\end{lemma}

Let $M$ denote a metric space, and 
consider a homeomorphism $f:M\to M$. 
For any starting point $x_0$, we often 
use subscript to 
denote the $f$-iterates of $x_0$, i.e., 
$x_n=f^n(x_0)$. 
For $\varepsilon>0$, 
we call an $\varepsilon$-pseudo periodic orbit (with period $n$), 
for a finite sequence of points 
$\{x^{(0)}, x^{(1)}, \cdots, x^{(n-1)} \}$, with the following properties.
\begin{align}
\text{dist} (f(x^{(i)}),x^{(i+1)}) & <\varepsilon, \text{ for any } i=0,\cdots,n-2, \text{ and, }\\
\text{dist} (f(x^{(n-1)}), x^{(0)}) & <\varepsilon.
\end{align}
Moreover, any point in an $\varepsilon$-pseudo periodic orbit as above, will be referred as 
an $\varepsilon$-pseudo periodic point. 

We say a point $p$ is $f$-recurrent, if 
there exists $n_k\to \infty$, such that $f^{n_k}(p)\to p$.
The following is a standard fact for all 
non-wandering dynamical systems on compact metric spaces $M$. 
We provide a proof for completeness.

\begin{lemma}\label{dense} Suppose $f:M \to M$ 
is non-wandering. Then, the set of $f$-recurrent points, 
denoted by $R(f)\subset M$, is dense. 
\end{lemma}
\begin{proof} 
Pick any open disk $B=B_0$. It suffices to show 
$B\bigcap R(f)\neq \emptyset$. 
Since $f$ is non-wandering, there is some $n_1$ such that 
$f^{-n_1}(B_0)\bigcap B_0 \neq \emptyset$. 
Then we can choose some small closed disk 
$B_1$, with radius 
smaller than $1$, such that 
$B_1 \subset f^{-n_1}(B_0)\bigcap B_0$. 
Note that every point in $B_1$ will return to 
$B_0$ at iterate $n_1$. 

Inductively, suppose we have found increasing integers $n_1<\cdots <n_k$,
 and closed disks $\{B_i\}_{i=1}^k$, such that for all 
 $i=1,\cdots, k$, $B_i$ has diameter smaller than $\frac 1i$, 
such that 
 $B_i \subset \text{Inter}( B_{i-1}\bigcap f^{-n_i}(B_{i-1}) )$. 
Then, we can choose some $n_{k+1}> n_k$ such that 
$B_k \bigcap f^{-n_{k+1}}(B_k)\neq \emptyset$, 
and some closed
disk $B_{k+1}$ with radius smaller than $\frac {1}{k+1}$, such that 
$B_{k+1}\subset  \text{Inter}(B_k \bigcap f^{-n_{k+1}}(B_k))$. 

Now for all $k\geq 1$, $B_k$ consists of points that will return to $B_{k-1}$ at time $n_k$. Now 
$\bigcap_{k\geq 1}B_k$ is a singleton, say, $\{x^\ast\}$. It follows that 
$f^{n_k}(x^\ast)\to x^\ast.$ 
\end{proof}

Denote by $\Bbb T^2$ the 
two-dimensional "flat" torus, 
whose universal covering space is $\Bbb R^2$, 
and 
let $\pi :\Bbb R^2 \to \Bbb T^2$ be the natural projection.
Let $\text{Homeo}_0(\Bbb T^2)$ denote
the set of homeomorphisms of $\Bbb T^2$ 
homotopic to the identity.
Then, denote by
$\text{Homeo}_{0,\text{nw}}(\Bbb T^2)$ and $\text{Homeo}_{0,\text{Leb}}(\Bbb T^2),$ 
the set of non-wandering and area-preserving 
homeomorphisms, respectively. Note that both are subsets of 
$\text{Homeo}_0(\Bbb T^2)$.
We also denote by 
$\widetilde {\text{Homeo}}_0(\Bbb T^2)$ 
(respectively, 
$\widetilde{\text{Homeo}}_{0,nw}(\Bbb T^2)$) the set of 
lifts of homeomorphisms 
from 
$\text{Homeo}_0(\Bbb T^2)$ 
(respectively, $\text{Homeo}_{0,nw}(\Bbb T^2)$) 
to the plane. 
 Similarly, for $r\geq1$, or $r=\infty$,
denote by 
$\text{Diff}_{r,\text{nw}}(\Bbb T^2)$ (respectively $\text{Diff}_{r,\text{Leb}}(\Bbb T^2)$)
the set of $C^r$ diffeomorphisms of $\Bbb T^2$, which are non-wandering
(respectively, area-preserving) and homotopic to the identity. Also, 
the sets of their lifts are denoted by $\widetilde{\text{Diff}}_{r,\text{nw}}(\Bbb T^2)$
and $\widetilde{\text{Diff}}_{r,\text{Leb}} (\Bbb T^2)$, respectively. 

 Usually, the choice of a lift 
 is not relevant for our purposes, 
 as can be seen in the next subsection.

We say a subset $K\subset \Bbb T^2$ is 
\textit{essential}
if there exists a non-trivial 
homotopy class, such that for any representative loop $\beta$ of it, 
$\beta\cap K\neq \emptyset$. A subset $K$ is called \textit{inessential}
if it is contained in a topological disk in $\Bbb T^2$. In this case, the complement 
of it is called \textit{fully essential}. We will need the following result. 
For more details about the above notations, see ~\cite{KoroleCalTal} and ~\cite{Strictly_Toral}.

\begin{lemma}\label{Bounded Invariant Disks}
[Theorem 6 in ~\cite{KoroleCalTal}]
Let $f\in \text{Homeo}_{0,\text{nw}}(\Bbb T^2)$ and suppose that the set of fixed points
 is inessential. Then there exists $M>0$, such that, 
for any $f$-invariant topological open disk $D$,
each connected component of $\pi^{-1}(D)$ has diameter bounded from above by $M$.
\end{lemma}

 \subsection{Misiurewicz-Ziemian Rotation Set} \hfill\\
 
Consider $\widetilde{f} \in \widetilde{\text{Homeo}}_0 (\Bbb T^2)$.
The foundations of the rotation theory in the torus were mainly developed by Misiurewicz and Ziemian in~\cite{MZ}. There, the following notion of a rotation set $\rho(\widetilde f)$ appears:
\begin{equation}\label{Misiurewicz_ziemann}
\rho(\widetilde f): =\{v =\lim_{k\to \infty}\frac{1}{n_k}(\widetilde f^{n_k}(\widetilde {z}_k)-\widetilde z_k),  \big| n_k\to \infty, \widetilde z_k\in \Bbb R^2, \text{whenever the limit exists}\}.
\end{equation}
\begin{Remark} 
For any $f \in \text{Homeo}_0(\Bbb T^2)$, and any $f$-invariant compact subset $K \subset \Bbb T^2$, 
$\rho(\widetilde f, K)$ is the rotation set of the map restricted to the 
invariant set $K$, which is defined similarly as in  
(\ref{Misiurewicz_ziemann}),
where the only difference is that, the points $\widetilde z_k$ in 
the expression are only allowed to be 
chosen in $\pi^{-1}(K)$. 
\end{Remark}

One can also define the point-wise rotation vector as follows. For any $z\in \Bbb T^2$,
\begin{equation}
\rho(\widetilde f,z): =\lim_{n\to \infty} \frac{1}{n} (\widetilde f^n(\widetilde z)-\widetilde z), \text{ when the limit exists}.
\end{equation}

Another important definition is as follows. 
Consider any $f$-invariant Borel probability measure $\mu$, 
and denote by $\rho_\mu(\widetilde f)=\int_{\Bbb T^2} \big( \widetilde f(\widetilde x)-\widetilde x \big) d\mu(x)$ 
the average rotation vector of the measure $\mu$ 
(note that 
the integrand in this expression is in fact a function  on $\Bbb T^2$). 
Define 
\begin{equation}
\rho_{\text{meas}}(\widetilde f):= 
\{\rho_\mu(\widetilde f) \big | \mu \text{ is a $f$-invariant Borel probability measure} \}.
\end{equation}
%In this paper, we will focus on the case when 
%$\rho(\widetilde f)$ is a non-trivial line segment. 
The following result gathers many important properties of these notions:

% \marginpar{\color{red}{I add Franks in the reference because the final statement was proved by Franks.}}
\begin{theorem}\label{rotationset}[See~\cite{MZ} and~\cite{MZ2}] 
For any $\widetilde f\in \widetilde{\text{Homeo}}_0(\Bbb T^2)$, 
$\rho_{\text{meas}}(\widetilde f)$ equals $\rho(\widetilde f)$, 
which is a compact and convex subset of $\Bbb R^2$. Moreover, every extremal point 
of the rotation set can be realized as the average rotation vector of some ergodic 
measure $\mu$. 
 \end{theorem}

\subsection{Bounded Deviations}\hfill\\
%The following lemma 
%(see~\cite{Atkinson}) is very useful 
%to do estimates for deviations of a homeomorphism 
%from its rotation set. 
% THE MEASURE DOES NOT NEED TO BE NON-ATOMIC 
%\begin{theorem}\label{Atkinson}[A particular case of Atkinson's lemma] 
%Consider $f\in \text{Homeo}_0(\Bbb T^2)$, 
%and let $\varphi:\Bbb T^2\to \Bbb R$ be a continuous function, 
%which induces a cocycle, defined by
%\begin{equation}
%\mathcal A^{(n)}(x):=\sum_{k=0}^{n-1}\varphi\circ f^n(x).
%\end{equation} 
%Assume $\mu$ is an ergodic $f$-invariant Borel probability measure, satisfying 
%$\int_{\Bbb T^2} \varphi(x)d\mu(x)=0$. 
%Then, for every $\mu$-typical point $x$, and $\varepsilon_k\to 0$, 
%there exist integer sequence $n_k\to +\infty$, such that,
%\begin{align}
 %\text{dist}_{\Bbb T^2}(f^{n_k}(x),x) &< \varepsilon_k, \text{ and }\\
   % |\mathcal A^{(n_k)}(x)| &< \varepsilon_k . 
%\end{align}
%\end{theorem}
For any non-trivial vector $w \in \Bbb R^2$, 
denote by $\text{pr}_w$ the projection of a vector 
along the $w$ direction.
\begin{equation}
\text{pr}_w: \Bbb R^2 \to \Bbb R, r \mapsto \langle r, w\rangle.
\end{equation} 
Next, we introduce the important notation of bounded deviations. 

\begin{definition}\label{definition_deviation}
Fix a non-trivial vector $w \in \Bbb R^2$. 
We say that $\widetilde f$ has bounded deviation along direction $w$ (from its rotation set $\rho(\widetilde f)$), 
if there exists $M>0$, such that for any $n\geq 0$ and any $\widetilde x\in \Bbb R^2$, 
\begin{equation}\label{one_direction}
\text{pr}_w(\widetilde f^n(\widetilde x)-\widetilde x-n \rho(\widetilde f))  \leq M.
\end{equation}
\end{definition}
\begin{Remark}
Note that with respect to this definition, 
having bounded deviation along $w$ and $-w$ are two different conditions. 
 \end{Remark}

The following statement essentially follows from Gottschalk-Hedlund theorem, 
see also ~\cite{Linearization} for a somewhat more elementary proof.

\begin{lemma}[Proposition A in~\cite{Linearization}]\label{jager}
Let $f \in \text{Homeo}_0(\Bbb T^2)$ preserve a minimal 
set $K \subset \Bbb T^2$. Suppose 
$\rho(\widetilde f, K) = \{ (\alpha,\beta)\}$ where $(\alpha,\beta)$
is totally irrational, and 
$\widetilde f\big|_{\pi^{-1}(K)}$ 
has bounded deviation along every direction. 
%Then, 
%there exists a \textit{regular}
%semi-conjugacy $\phi:K\to \Bbb T^2$ between 
%$(K, f)$ to $(\Bbb T^2, R_{(\alpha,\beta)})$, 
Then, there exists a continuous surjective map
$\phi:K\to \Bbb T^2$,  homotopic to 
the inclusion, satisfying that
\begin{equation}
\phi\circ f\big|_K=R_{(\alpha,\beta)}\circ \phi,
\end{equation}  
where $R_{(\alpha,\beta)}$
is the 
rigid rotation on $\Bbb T^2$. 
\end{lemma}

The following result establishes 
the bounded deviation property in 
the perpendicular direction when the rotation set is 
the special one we are interested in. 
%The statement is a combination of the 
%main theorems 
%of~\cite{Guelman} and ~\cite{Guilherme_Thesis},
 %in the case 
%when 
%the line segment is of rational and irrational slope, respectively.

\begin{theorem}[\cite{Guilherme_Thesis}]\label{Fabio_Bounded}
Suppose $\widetilde f
\in \widetilde{\text{Homeo}}_0(\Bbb T^2)$,
and $\rho(\widetilde f)$ is the segment from 
$(0,0)$ to a totally irrational point 
$(\alpha,\beta)$. 
Then $\widetilde f$ has bounded deviation 
along the perpendicular directions
$(-\beta,\alpha)$ and $(\beta,-\alpha)$.
\end{theorem}

\subsection{Some Fundamental Tools in Topological Dynamics}\hfill\\
Let $f: \Bbb R^2 \to \Bbb R^2$ denote an orientation-preserving
 homeomorphism. 
 A properly
 embedded oriented line $\gamma:\Bbb R\to \Bbb R^2$
 is called a \textit{Brouwer line}, if $f(\gamma(\Bbb R))$ and $f^{-1}(\gamma(\Bbb R))$ belong to different  
connected components of the complement of $\gamma(\Bbb R)$.
 We also abuse notation by writing 
 $\gamma=\gamma(\Bbb R).$ 
 We call these components 
 \emph{the right of $\gamma$} and \emph{the left of $\gamma$},
 and denote them by $\mathcal R(\gamma)$ and $\mathcal L(\gamma)$, 
 respectively. 
  
The following result is usually attributed to
 Brouwer (see also~\cite{Brown_newproof}), we 
refer to~[\cite{Franks_PB}, Proposition 1.3] 
for a very useful generalization. 
Here we state a weaker version, which is 
sufficient for our use.
\begin{lemma}\label{Brouwer}
Suppose $f:\Bbb R^2\to \Bbb R^2$ is an orientation-preserving homeomorphism. If 
there exists a topological open disk $U$, 
such that $f(U)\bigcap U =\emptyset$, and 
for some $k\geq 2$, $f^k(U)\bigcap U\neq \emptyset$,
then, if the fixed points are isolated,  $f$ admits a fixed point with
positive topological index.
\end{lemma}

\begin{definition} Let
$f: (U,p)\to (V,p)$ 
denote a local homeomorphism, 
where $U$ and $V$ are two open neighbourhoods of 
an isolated fixed point $p \in \Bbb R^2$. Choose 
a small disk $D\subset U$, whose boundary is $\beta=\partial D$. Define 
$g:\beta \to S^1$, such that $g(x)=\frac{x-f(x)}{\|x-f(x)\|}$. 
The topological index of $f$ at the point $p$ is
 defined as the degree of the map $g$, denoted as $\text{Index}_f(p)$.
\end{definition}
The following 
is a consequence of the classical result usually referred to as 
Lefschetz fixed-point formula. 
(See for example Theorem 8.6.2 in \cite{Katokbook}).

\begin{lemma}\label{Lefschetz} 
Let $f \in \text{Homeo}_+^0(\Bbb T^2)$
and assume all fixed points are isolated. 
Then 
\begin{equation}
\sum_{p\in \text{Fix}(f)} \text{Index}_f(p)=0.
\end{equation}
\end{lemma}

%TRY TO AVOID USING THE WORDS EASY, SIMPLE, IMPORTANT, ETC. MY 
%EXPERIENCE IS THAT REFEREES DONT LIKE THEM.

\subsection{Generic Conditions}\hfill\\

Recall 
$\text{Diff}_{r,\text{nw}}(\Bbb T^2)$ is 
the set of 
non-wandering $C^r$ 
diffeomorphisms of $\Bbb T^2$, which are homotopic to the
identity, and $\text{Diff}_{r,\text{Leb}}(\Bbb T^2)\subset \text{Diff}_{r,\text{nw}}(\Bbb T^2)$
is the subset of area-preserving ones. 
Below, we say $p$ is a 
periodic \textit{saddle-like} point for $f$ if there is $n>0$, 
such that 
$p$ is a $f^n$-fixed point, 
and with respect to $f^n$, 
the dynamics near $p$ is obtained by gluing a finite number of topological saddle-sectors, see \cite{Dumortier}. 
% USE THE CORRECT REFERENCE. 
As usual, 
we denote by $W^u(p)$ (respectively, $W^s(p)$)
the union of $p$ and all the unstable branches at $p$ (respectively, the
stable branches at $p$). 

\begin{definition}\label{Gr_definition}
For any $r\geq 1$ or $r=\infty$, 
define 
$\mathcal G^r \subset \text{Diff}_{r,\text{nw}}(\Bbb T^2)$ 
as the subset of diffeomorphisms $f$
satisfying the following two conditions: 
\begin{enumerate}
\item  if $p \in \Bbb T^2$ is an $n$-periodic point, then  
$Df^n(p)$ does not have $1$ as an eigenvalue. 
\item  $f$ does not have saddle connections.
\end{enumerate}
\end{definition}

\begin{Remark}\label{these_conditions_are_generic}
By Theorem 3 (iii) and Theorem 9 of~\cite{Robinson},
for all $r \geq 1$, the set
$\mathcal G^r \bigcap \text{Diff}_{r,\text{Leb}}(\Bbb T^2)$
contains a residual subset of $\text{Diff}_{r,\text{Leb}}(\Bbb T^2)$.
Thus the above conditions are generic for area-preserving 
diffeomorphisms. 
\end{Remark}

\begin{definition}\label{topological_transverse}
Let $f:S\to S$ be a $C^1$ 
diffeomorphism on a closed surface $S$. Let $p,q$ be two 
periodic saddle-like points. 
We say that $W^u(p)$ and $W^s(q)$ intersect 
at a point $w$ in a topologically transverse way,  
if there exists an open disk $B=B(w,\delta)$
with some radius $\delta>0$ such that  
(denote by $\alpha,$ respectively $\beta$, 
the connected component of $W^u(p)\bigcap B,$ respectively $W^s(q)\bigcap B$, 
both containing the point $w$):

\begin{enumerate}
\item $B \backslash \alpha= B_1 \sqcup B_2$, which is a disjoint union.
\item $\beta\backslash \{w\}= \beta_1\sqcup \beta_2,$ which is a disjoint union and
$\beta_1 \subset B_1\bigcup \alpha$, $\beta_2 \subset B_2\bigcup \alpha$, 
with $\beta_1  \not \subset \alpha$ and $\beta_2\not\subset \alpha$. In other words, 
$\beta_1$ intersects $B_1$ and $\beta_2$ intersects $B_2$.
\end{enumerate}
\end{definition}

\begin{Remark}\label{saddle_like_case_noncontracible} The radius of the disk, 
$\delta>0,$ might have to be taken strictly away from $0$, because 
the connected component of $\alpha\bigcap\beta$ containing $w$ could be a non-trivial arc. 
\end{Remark}

The following lemma will be useful for obtaining non-contractible periodic orbits, i.e., 
those orbits with non-zero rotation vectors. 
\begin{lemma}\label{create_topological_horseshoe}[Lemma 0 in \cite{Salvador_Growth}]
Suppose 
$\widetilde f\in \widetilde{ \text{Diff}}_{r}(\Bbb T^2)$ has
% THE \WIDETILDE WAS MISSING, CHECK IF IT APPEARS CORRECT IN THE PDF
 a hyperbolic periodic saddle point $\widetilde q$, and suppose 
$W^u(\widetilde q)$ and $W^s(\widetilde q)+(a,b)$ intersect in a
 topologically transverse way, for some integer vector $(a,b).$ 
Then there exists some integer $N>0$ such that the diffeomorphism 
$\widetilde g: = \widetilde f^{N}-(a,b)$ 
admits a fixed point $\widetilde p$. Thus, $\rho(\widetilde f, \pi(\widetilde p) )=(\frac{a}{N},\frac{b}{N})$.
\end{lemma}
\begin{Remark}\label{saddle_like_points_horseshoe} 
Following exactly the same proof as in \cite{Salvador_Growth}, 
the conclusion is also true when the periodic point $\widetilde q$ has topological index $0$, 
and admits exactly one stable and one unstable branch, whose local dynamics is 
described in Figure 1. 
\end{Remark}
\subsection{A Broader Class of (Non-generic) Diffeomorphisms}\hfill\\
  
  The following definition appears as Definition 1.6 on page 12 of~\cite{Dumortier}. 
  Such a study is  based on the important work of Takens
  (See Theorem 4.6 in~\cite{Takens}). 
  Although all the theory in \cite{Dumortier} was stated  
for $C^\infty$ maps, for the results we consider below,  
  there is no substantial difference in the $C^r$ case.
  
    \begin{definition}\label{Lojasiewicz_type} Let $f: (U,p)\to (W,p)$ be 
  a local planar $C^r$-diffeomorphism with an isolated fixed point $p$. 
   Assume all eigenvalues of $Df(p)$ belong to the unit circle and 
   let $S$ be the semi-simple part of $Df(p)$ in its Jordan 
  normal form. Then up to a $C^r$-coordinate change, 
  there exists a formal $C^r$ vector field $X$, invariant under $S$,
   such that, the $r$-jet of $f$ and 
  the $r$-jet of $S \circ X_1$ coincide at $p$, where $X_1$ is the time-1 map of the formal flow generated by $X$. 
  We say $f$ is \textit{of Lojasiewicz type} at $p$, if the following condition holds:
  \begin{itemize}
%  \item let $S$ be the semi-simple part of $Df(p)$ in its Jordan 
 % normal form, 
  %then up to a $C^r$-coordinate change, 
  %there exists a formal $C^r$ vector field $X$, invariant under $S$,
   %such that, the $r$-jet of $f$ and 
  %the $r$-jet of $S \circ X_1$ coincide at $p$, where $X_1$ is the time-1 map of the formal flow generated by $X$. 
  \item there exists an integer $k\leq r$ and constants $C,\delta>0$, 
  such that, for any $z$ satisfying $\|z-p\|\leq \delta,$ then
    \begin{equation}
  \|X(z)\| \geq C \|z-p\|^k.
  \end{equation}
  \end{itemize}
   \end{definition}
  
  The following result was essentially obtained in Section 2 
of ~\cite{Addas_Calvez}.
\begin{lemma}\label{section_two_lemma}
Assume $f\in \text{Diff}_{r,\text{nw}}(\Bbb T^2)$ has an isolated 
fixed point $p$ and the topological index of $f$ at $p$ is zero.
Also suppose that if both eigenvalues of $Df$
at $p$ lie in the unit circle, then $f$ is of 
Lojasiewicz type at $p.$
Then, 
there exists exactly one stable  
and one unstable branch at $p$. 
Moreover, 
the local dynamics can be precisely described, 
see Figure~\ref{the_local_situation}.
%
% PLEASE INCLUDE THE PICTURE
%
%(see for example 
%Picture 1 of 
%the paper~\cite{Addas_Calvez}).
%In particular, there exists a local continuous chart 
%$\psi: B \to [-1,1]^2$, such that,
%the following holds. 
%\begin{enumerate}
%\item 
%for each  $y \in [-1,1]$,
% $\psi^{-1}([-1,1]\times \{y\})$ 
%contains some translation arc, i.e., there is some subarc $\lambda$
%whose endpoints are some point $y$ and its image $f (y)$.
%\item $\psi$
% is a homeomorphism between  $B$ and $[-1,1]^2$, and $ \psi(p)=(0,0)$. 
 %\item for any $ z\in (B\cap f^{-1}(B) )\backslash\{p\}$
% \begin{equation}
 %\text{pr}_1\circ \psi \circ f(z)    > \text{pr}_1\circ \psi(z). 
% \end{equation}
%\item for any $z\in B\bigcap f^{-1}(B)$
%\begin{equation}
%\text{pr}_2 \circ \psi \circ f(z)  = \text{pr}_2 \circ \psi (z).
%\end{equation}
%\end{enumerate}
\end{lemma}

\begin{figure}[!ht]
\begin{center}
\includegraphics[width=5cm]{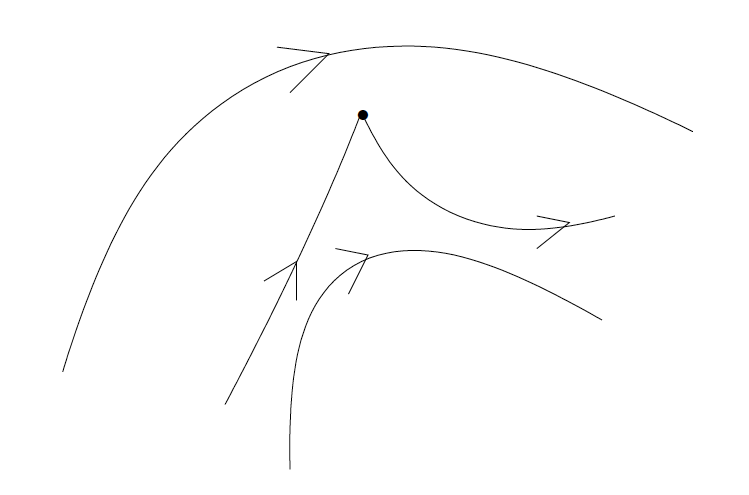}
\caption{The Local Dynamics Around a Fixed Point}
\label{the_local_situation}
\end{center}
\end{figure}
\begin{proof} 
Consider the linear transformation $Df(p)$, which has positive determinant. 
If $1$ is not an eigenvalue of $Df(p)$, then $p$ is either an hyperbolic fixed saddle point,  
an elliptic fixed point (that is, both eigenvalues are in the unit circle and not real),
or $-1$ is an eigenvalue. 
In all the above possibilities, $p$ has topological index equal to $-1$ or $1$. 
As the index at $p$ is zero,
the above cases do not happen. 
If the two eigenvalues of $Df(p)$ 
are $1$ and some $a>0$ with $a\neq 1$, then as an application of center manifold 
theory (see \cite{Carr}), we get that  
$p$ can be a topological saddle, a topological sink (or source), or a saddle-node.  
Since by assumption $p$ has topological index $0$, it must be a saddle-node. 
In this case, $p$ has two saddle sectors, and one sector which is 
either contracting or expanding, 
a contradiction with the non-wandering condition. 
For more details, see Proposition 6 from~\cite{Addas_Calvez}, 
as well as~\cite{Carr}.

Thus, $Df(p)$ must have both eigenvalues equal to $1$. 
The rest of the proof follows exactly the same lines from 
the argument in Section 2 of \cite{Addas_Calvez}.
\end{proof}

%\begin{Remark}\label{Saddle_like} 
%An $f$-fixed point as above is said to be a \textit{saddle-like} fixed point. It has exactly 
%two saddle sectors, and exact one stable branch and one unstable branch. 
%\end{Remark}

    \begin{definition}\label{generic_in_generic}
    For any $r\geq 1$ or $r=\infty$, 
define $\mathcal K^r \subset \text{Diff}_{r,\text{nw}}(\Bbb T^2)$ to be the subset of 
diffeomorphisms $f$ satisfying 
the following four conditions:
\begin{enumerate}
\item  for all $n>0$, there 
are at most finitely many $f^n$-fixed points. 
%and if $p$ is such a point, and $1$ is an eigenvalue of $Df^n(p),$ then it has 
%multiplicity $2$ and $f^n$ is 
%of Lojasiewicz type at $p$.  
\item  for any $f^n$-fixed point $z$, of prime period $n$, if $1$ is an 
eigenvalue of $Df^n(z),$ then it has multiplicity $2$ and $f^n$ is 
of Lojasiewicz type at $z$. Moreover, in this case the $\text{Index}_{f^n}(z)$ 
is zero, and so Lemma~\ref{section_two_lemma} implies that 
the local dynamics near $z$ is given by Figure 1. For families of maps, this 
situation corresponds to the birth or death of periodic points. 
\item for any $f^n$-fixed point $w$, of prime period $n$, if $-1$ is an 
eigenvalue of $Df^n(w),$ then $\text{Index}_{f^n}(z)$ is $1.$ 
For families, this situation corresponds to a period-doubling bifurcation.
\item there are no connections 
between saddle-like periodic points.
\end{enumerate}
    \end{definition}

 $\mathcal K^r\backslash \mathcal G^r$
 can be thought as 
a sort of set of typical diffeomorphisms 
in the complement of $\mathcal G^r$.
The definition can also be justified as follows. 
Let $\mathcal F$  
denote some one-parameter 
 $C^r$-generic family of area-preserving diffeomorphisms, 
  \begin{equation}
  \mathcal F:=\{f_t\}_{t\in[a,b]} \subset \text{Diff}_{r,\text{Leb}}(\Bbb T^2), \text{ for some } r\geq 1.
  \end{equation}
 The following statement
  is a combination of results from ~\cite{Salvador_Growth} and ~\cite{Addas_Calvez}. 
  The proofs were based on previous results contained in~\cite{Meyer} and~\cite{Dumortier}.
  \begin{lemma}\label{generic_family}[Section 1.3.3 of~\cite{Salvador_Growth} and 
  Section 2 of ~\cite{Addas_Calvez}] 
  For such a generic 
  $C^r$-family $\mathcal F$ as above,
  % = \{f_t\}_{t\in [a,b]}\subset \text{Diff}_{r,\text{Leb}}(\Bbb T^2)$, 
   % $f_t\in \mathcal K^r$ for all $t\in [a,b]$. In other words, for a $C^r$-generic
   %one parameter family $f_t$ of area-preserving diffeomorphisms of the torus homotopic 
   %to the identity, 
   all maps $f_t$ belong to $\mathcal K^r$; in particular,  
such a family never has saddle-like connections (tangencies are of course allowed),
and with respect to periodic points, the only allowed degeneracies 
for a certain parameter $t$ are, period-doubling bifurcations (item (3) above) 
and the one which appears in item (2), which as we already said, is related to 
the birth or disappearance of periodic points for families of maps.
  \end{lemma}

The next result gives a perturbation consequence based on 
the local dynamics near fixed points. Let 
$Fix(\widetilde{f})=\{z\in \Bbb T^2:\forall 
\widetilde{z}\in \pi ^{-1}(z),$ $
\widetilde{f}(\widetilde{z})=\widetilde{z}\}.$

\begin{prop}[Proposition 9 of \cite{Addas_Calvez}]\label{AC} 
Suppose $\widetilde f \in \widetilde{\text{Homeo}}_0 (\Bbb T^2)$. 
Assume $Fix(\widetilde{f})$ is finite and for any 
$z_0\in \text{Fix}(\widetilde{f})$,  
there exists a local continuous chart $\psi:U\to \Bbb R^2$, such that, for any 
$z\in ( U\cap f^{-1}(U) )\backslash \{z_0\}$, 
$\text{pr}_1\circ \psi\circ f(z) > \text{pr}_1 \circ \psi(z)$. 
 Then, there exists $\varepsilon>0$, 
  such that for
  any any $\widetilde g \in  \widetilde{\text{Homeo}}_0 (\Bbb T^2)$
  with $\text{dist}_{C^0}(\widetilde g,\widetilde f)<\varepsilon$,
  $(0,0)\notin \text{int}(\rho(\widetilde g))$.
\end{prop}    
\subsection{Prime Ends Rotation Numbers}\hfill\\

For an open topological disk $D$ contained in some surface, 
one can attach an artificial circle, called the \textit{prime ends circle},
denoted by $b_{\text{PE}}(D)$. 
Moreover, 
the prime ends circle $b_{\text{PE}}(D)$ can 
be be topologized, 
such that the union 
$D\bigcup b_{\text{PE}}(D)$ is homeomorphic to 
the standard closed unit disk in $\Bbb R^2$. 
We call it the \textit{prime ends disk}, and 
this procedure is referred to as 
 \textit{prime ends compactification}.
This is the beginning of Carath\'eodory's prime ends theory, and we refer 
to~\cite{Mather_topological}, \cite{Mather_invariant} and \cite{PE_01} for more details. 

If $f$ is a 
homeomorphism on the closure of an open topological 
disk $D$ into itself, then 
$f$ extends to the unit disk $D\bigcup b_{\text{PE}}(D)$, 
so it induces a 
 homeomorphism on the 
 prime ends circle $b_{\text{PE}}(D)$. 
 Then, the dynamics restricted to this circle defines 
 a rotation number, 
called the \textit{prime ends rotation number}, and
denoted by $\rho_{\text{PE}}(f,D)$.
The following lemma is a combination of important results from 
several papers.

\begin{lemma}\label{primeends} 
Suppose $\widetilde h$
 is a lift of some $h\in \text{Diff}_{r ,\text{nw}}(\Bbb T^2)$ 
 satisfying
the following properties: 
\begin{enumerate}
\item There are at most finitely many  
$h^n$-fixed points for all $n\geq 1$;
%and if both eigenvalues of $Dh^n$ at such a point belong to the unit circle, then
%it is of Lojasiewicz type for $h^n$.  
\item For all $n\geq 1$, and any 
$h^n$-fixed point $p$, of prime period $n$,
 \begin{enumerate}
 \item either
none of the eigenvalues of $Dh^n(p)$ is equal to $1$,
\item or, if one 
of the eigenvalues of  $Dh^n(p)$ is $1$, then the topological index of $h^n$ at 
$p$ is zero and actually, both eigenvalues are equal to $1$ (see the  
proof of Lemma \ref{section_two_lemma}). Moreover, $p$ is of Lojasiewicz type for $h^n$.
\end{enumerate}
\end{enumerate}
Let $K\subset \Bbb R^2$ be an $\widetilde h$-invariant 
continuum and let $O$ denote 
an $\widetilde h$-invariant 
connected component of $K^c.$
Write $\rho_{\text{PE}}(\widetilde {h}, O)$ to denote  the prime ends rotation number of $\widetilde h$ restricted to $O$. 
Then the following statements are true: 
\begin{enumerate}[i)]
\item 
If $\rho_{\text{PE}}(\widetilde {h}, O)$ 
is rational and $O$ is bounded, or $\rho_{\text{PE}}(\widetilde {h}, O)$ is zero, then
 $\partial O$ contains only saddle-like $\widetilde {h}$-periodic points, and connections between these saddle-like periodic points.  
\item If $O$ is not equal to $K^c$ and 
$\rho_{\text{PE}}(\widetilde {h}, O)$ is irrational, 
then there is no $\widetilde h$-periodic point in $\partial O$.
\end{enumerate}
\end{lemma} 

\begin{proof}[Sketch of the proof] Let us show item (i). 
Assume $\rho_{\text{PE}}(\widetilde {h}, O)=\frac pq$ 
which is in reduced form. 
Then, as $h$ is non-wandering, a prime chain corresponding to a $q$-periodic prime end $\widehat{z}$ has the property that each of its crosscuts must intersect its image under $\widetilde{h}^{q},$ otherwise the corresponding crossections would contain wandering domains, even in the torus (this argument goes back to Cartwright-Littlewood, see for instance proposition 2.1 of \cite{Franks_LeCalvezRegions}). So, the principal set of $\widehat{z}$ is made of $\widetilde{h}^{q}$-fixed points, something that implies the first assertion of item (i), i.e., 
there exists some $\widetilde{h}^{q}$-fixed point $z\in \partial O$. 

The  main results of \cite{PE_02} (see also Theorem 1.2, Corollary 1.3 and Theorem 1.4 of the report \cite{ICMKoro} from ICM 2018), imply:
\begin{enumerate}
\item if $q=1$, then all $\widetilde{h}$-periodic points $z\in \partial O$ are fixed and the eigenvalues of $D\widetilde{h}(z)$
are both real and positive;
\item if $O$ is bounded, for any value of $p/q$, all $\widetilde{h}$-periodic points $z\in \partial O$ have prime period $q$ and 
the eigenvalues of $D\widetilde{h}^q(z)$ are both real and positive;
\end{enumerate}
 So, as $h$ is non-wandering, a periodic point $z\in \partial O$ is either an hyperbolic saddle or both eigenvalues of $Dh^q(z)$ are equal to $1$ and $z$ has topological index $0$. In this way, Lemma \ref{section_two_lemma} implies that $z$ is a saddle-like periodic point, either a hyperbolic saddle or a point with zero index and local dynamics as is Figure~\ref{the_local_situation}.

With this local description, 
Theorem 5.1 in~\cite{Mather_invariant} implies that 
the boundary $\partial O$ contains connections between  
saddle-like periodic points, as described above.
We should remark that although in reference \cite{ICMKoro}, 
most statements assume preservation of area 
and boundedness of $O$ as a planar subset, 
the arguments therein only use the fact that
the dynamics is non-wandering 
restricted to a small neighbourhood 
of the compact set $K$. 
This completes the proof of 
 item (i). 
 
Item (ii) is a direct consequence of Theorem C of \cite{PE_01}.
\end{proof}

\section{Perturbations for Homeomorphisms with Unbounded Deviation}
\label{proofs_unbounded}

In this section and in the next, 
for any $w\in \Bbb R^2$, we will write  $[w]$ to denote an integral vector which is the closest to $w$.
Also, a condition assumed in all the theorems proved here 
is unbounded deviation for a fixed direction, along which,
we want our rotation set to grow. 

\begin{theorem}\label{instability} 
Let $\widetilde f \in \widetilde{\text{Homeo}}_{0,\text{nw}}(\Bbb T^2)$ whose 
rotation set $\rho(\widetilde f)$ is a line segment from $(0,0)$
to some $(\alpha,\beta) \in \Bbb R ^2$ which is totally irrational.
Assume
$\widetilde f$ has unbounded deviation along the direction 
$(\alpha,\beta)$. 
Then $\widetilde f$ can be $C^0$-approximated by 
$\widetilde g \in \widetilde {\text{Homeo}_0}(\Bbb T^2)$ 
such that $\rho(\widetilde g)$ has non-empty interior and
 \begin{equation}
 \rho(\widetilde f)\backslash \{(0,0)\}
  \subset \text{Int}( \rho(\widetilde g) ).
 \end{equation} 
\end{theorem}

When the rotation set is as above, 
we can also study a similar situation around 
the other endpoint, $(0,0)$. 

%\marginpar{the statement of this theorem changed back to the usual
%bounded deviaition for orbits. }

\begin{theorem}\label{case_of_0}
Let $\widetilde{f}$ and $\rho(\widetilde f)$ be as in Theorem~\ref{instability}.
Assume $\widetilde {f}$ admits unbounded deviation along $-(\alpha,\beta)$.
Then $\widetilde {f}$ can be $C^0$-approximated 
by $\widetilde g \in \widetilde{\text{Homeo}}_0(\Bbb T^2)$,
such that $(0,0) \in \text{Int}( \rho(\widetilde g) )$. 
\end{theorem}
%\marginpar{maybe here we explain precisely what we mean by epsilon pseudo orbit with unbounded 
%deviations. Or maybe we dont, I think it is clear. You decide. 
%\textcolor{red}{I prefer not to explain, it should be clear... I also modified a bit the statement of 3.2.}}

\subsection{Some Preparations}\label{forms}\hfill \\
Given the totally irrational vector $(\alpha,\beta)$, define 
\begin{align}
L_0 & : y=\frac{\beta}{\alpha}x. \\
L_1 & : \alpha x+\beta y=\alpha^2+\beta^2,
\end{align}
which are straight lines along the directions 
$(\alpha, \beta)$, $(-\beta,\alpha)$, respectively, and intersecting at 
the point $(\alpha,\beta)$. 
Also define the four connected components of the complement of 
$L_0 \bigcup L_1$ in $\Bbb R^2$. See 
Figure~\ref{the_four_regions}.
\begin{align}
\Delta_0         & =   \{ \widetilde z \in \Bbb R^2   \big|  
\text{pr}_{(-\beta,\alpha)} (\widetilde z) >0
 \text{ and } \text{pr}_{(\alpha,\beta)} \big(\widetilde z - (\alpha,\beta) \big)  > 0 \}. \label{region1}\\
\Delta_1         & =   \{ \widetilde z  \in \Bbb R^2  \big| 
\text{pr}_{(-\beta,\alpha)} (\widetilde z)<0 \text{ and }  \text{pr}_{(\alpha,\beta)}\big( \widetilde z-(\alpha, \beta) \big) < 0\}. \\
\Omega_0      & =  \{ \widetilde z  \in \Bbb R^2  \big|  \text{pr}_{(-\beta,\alpha)} (\widetilde z)<0 \text{ and } \text{pr}_{(\alpha,\beta)} \big(\widetilde z-(\alpha,\beta) \big) >0 \}. \\
\Omega_1      & =  \{ \widetilde z \in \Bbb R^2   \big|  
\text{pr}_{(-\beta,\alpha)}(\widetilde z) >0 \text{ and } \text{pr}_{(\alpha,\beta)} \big(\widetilde z -(\alpha,\beta) \big) < 0 \}.\label{region4} 
\end{align}

\begin{figure}[!ht]
\begin{center}
\begin{picture}(165,165) 

\put(0,30){\vector(1,0){150}}
\put(30,0){\vector(0,1){150}}
\put(10,10){\line(1,1){120}}
\put(60,120){\line(1,-1){60}}

\put(60,85){$\Omega_1$}
\put(85,110){$\Delta_0$}
\put(85,60){$\Delta_1$}
\put(90,90){\circle*{3}}

\put(80,76)
{ \scriptsize $(\alpha,\beta)$}
\put(110,85){$\Omega_0$}
\end{picture}
\end{center}
\caption{The Four Regions.}\label{the_four_regions}
\end{figure}

%Below we say $\phi:\Bbb T^2\to \Bbb T^2$ is supported on some open set 
 %$U$ if and only if $\phi \big|_{U^c}=id\big|_{U^c}$. 
\begin{lemma}\label{firstclosing} 
Let $\{x^{(0)}, \cdots, x^{(n-1)}\}$ 
be an $\varepsilon$-pseudo periodic orbit of some $f\in \text{Homeo}_0(\Bbb T^2)$. 
Then, $f$ can be $C^0$-$\varepsilon$-perturbed to 
$g$, also in $\text{Homeo}_0(\Bbb T^2)$
 such that $g^n(x^{(0)})=x^{(0)}$. Moreover, if $f$ preserves area, then so does $g$.
\end{lemma}

\begin{proof}
See~\cite{Franks_ETDS_8}. For the area-preserving case, just note that $g$ is obtained 
from $f$ by a series of perturbations supported in finitely many disjoint disks. And 
it is well-known that these perturbations, which are the identity in the boundary of 
each disk, can be chosen as area-preserving themselves. 
%By assumption, 
%for any $i=0,\cdots, n-2$, 
%$d(f ( x^{(i)} ), x^{(i+1)})<\varepsilon$, 
%and $d(f( x^{(n-1)}), x^{(0)})<\varepsilon$. 
%Then, by a classical result of Oxtoby (see \cite{Oxtoby_diameters}), 
%we can choose disjoint paths 
%$\{\gamma_i\}_{i=0}^{n-1}$, with the following properties. 
%\begin{enumerate}
%\item for any $i=0,\cdots, n-2$, $\gamma_i$ connects $f(x^{(i)})$ and 
%$x^{(i+1)}$, and $\gamma_{n-1}$ connects $f( x^{(n-1)} )$ and $x^{(0)}$. 
%\item the diameter of $\gamma_i$ is smaller than $\varepsilon$ for all $i=0,\cdots,n-1$ .
%\item $\{\gamma_i\}_{i=1}^k$ are pairwise disjoint. 
%Moreover, each $\gamma_i$ intersects the point set $\{x^{(0)}, \cdots, x^{(n-1)}\}$
% only at its two endpoints.
%\end{enumerate}
%Then, for any $0=1,\cdots, n-1$, 
%we can define $\phi_i$, with the following properties. 
%\begin{enumerate}
%\item each $\phi_i$ is supported on 
%a sufficiently small neighborhood $U_i$ of 
%$\gamma_i$, with $\text{diam}(U_i)<\varepsilon$.
%\item $\{U_i\}_{i=1}^k$ are pair-wise disjoint.
%\item  for any $i=1,\cdots,n-2$, $\phi_i( f( x^{(i)}) =x^{(i+1)}$;
 %         $\phi_{n-1} \big( f(x^{(n-1)}) \big)=x^{(0)}$. 
%\end{enumerate}
%Now we can define $g=\phi_{n-1}\circ \cdots \circ \phi_1 \circ f$, and the proof is completed.
\end{proof} 

 %\begin{lemma}\label{many-closing}[$C^0$-Closing Lemma for Several Orbits] 
%Suppose that there are $k$ $\varepsilon$-pseudo periodic orbits, with lengths $n_1,\cdots,n_k$,
%as follows. 
%\begin{equation}\label{k_orbits} 
%A_1=\{x^{1,(0)}, \cdots, x^{1,(n_1-1)}\},  
%\cdots, 
%A_k=\{x^{k,(0)}, \cdots, x^{k,(n_k-1)}\}.
%\end{equation} 

%Assume all these points appearing in these sequences are point-wise disjoint. Then, 
%there exists $g\in \text{Homeo}(\Bbb T^2)$ with $\text{dist}_{C^0}(g,f)<\varepsilon$, 
%such that,
%for any $i=1,\cdots, k$, and $j=0,\cdots, n_i-2$,
%$g(x^{i,(j)}) = x^{i,(j+1)}$, and $g(x^{i,(n_i-1)})=x^{i,(0)}$.
%\end{lemma}

%\begin{proof}
%As in the proof of %Lemma~\ref{firstclosing}, 
%we want to choose disjoint paths 
%connecting the image of one point %to another point, 
%in order to close all pseudo-orbits. 
%We hope for all the connecting paths we find to be pairwise disjoint, 
%Note that perturbing the pseudo-orbits in an arbitrarily small way, 
%we can assume that all the points %above are pair-wise disjoint and  
%reduce to the case 
%when all the points above are pair-wise disjiont. 
%for all different $i,i'$, the points
%$f(x^{i,(n_i-1)})$ and $x^{i',(0)}$ %do not coincide. 

%So, using the same ideas from the %proof of the previous lemma, we can %find disjoint perturbations, 
%such that their composition closes %all $k$-pseudo-periodic orbits.
%\end{proof}

Recall that Theorem~\ref{Fabio_Bounded}
says that $\widetilde f$ 
has bounded deviation along perpendicular directions.
Based on this, the following lemma gives small displacements along the perpendicular directions, for some chosen iterates. 

\begin{lemma}[Small Displacement]\label{displacement} 
	%Recall we denote $v= (\alpha,\beta)$ and $v^\perp=(-\beta,\alpha)$.
	Let $w$ be either $(-\beta,\alpha)$ or $(\beta,-\alpha).$
	Then, for any $\delta>0$ and any $\widetilde z_0\in \Bbb R^2$,  
	there exists $n_0$ such that, 
	\begin{equation}\label{bounded_displacement}
		\text{pr}_{w} 
		(\widetilde f^n (\widetilde{z}_0) -\widetilde f^{n_0} (\widetilde {z}_0) ) 
		<   \delta, \text{ for any } n>n_0. 
	\end{equation}
\end{lemma}

\begin{proof} 
	Observing  Theorem~\ref{Fabio_Bounded},
	we can choose $n_0 \geq 1$, with 
	\begin{equation}
		\text{pr}_{w} (\widetilde f^{n_0} 
		(\widetilde {z}_0)- \widetilde {z}_0) > 
		\sup_{n\geq 1} \{ 
		\text{pr}_{w} (\widetilde f^n (\widetilde {z_0})- \widetilde {z_0}) \} - \delta.
	\end{equation}
	Then, for any $n>n_0$, 
	(\ref{bounded_displacement}) follows immediately.
\end{proof}

\subsection{The Irrational Endpoint}
\begin{proof}[Proof of Theorem~\ref{instability}]
%\marginpar{For any recurrent point contained in the support of an ergodic 
%measure with rotation vector (alfa,beta), there is a nearby recurrent point, also
%contained in the support of the measure, which satisfies what you want. This is
%the content of my old paper. \textcolor{red}{Thanks for this, I corrected this statement. Moreover, I combine the statement into only one case. It is not necessary to split into two cases. it simplifies the argument.}}
Fix $\varepsilon>0$. By Theorem 1 of~\cite{Salvador}, 
for any ergodic measure with average rotation vector $(\alpha,\beta)$, around 
a typical point which is $f$-recurrent, there is some point $y$ and a positive integer $n$, such that 
\begin{align}
& \text{dist}_{\Bbb T^2}(f^n(y), y)<\varepsilon.\label{epsilon_close}\\
& \text{pr}_{(\alpha,\beta)}([ \widetilde f^n(\widetilde y)  - \widetilde y] -n (\alpha,\beta))  > 0 . \label{epsilon_away}
\end{align}
Considering Figure~\ref{the_four_regions}, the last estimate implies that
\begin{equation}\label{Delta0Omega0}
\frac {1}{n} [\widetilde f^n(\widetilde y) -\widetilde y] \in \Omega_0 \bigcup \Delta_0.
\end{equation}
%There are two cases to consider.\\
%\noindent{\bf Case One.} 
%There exist two points $y$ and $y'$, with disjoint orbits, so that both (\ref{epsilon_close}) and 
%(\ref{epsilon_away}) are satisfied, and moreover, the vectors as in (\ref{Delta0Omega0}),
% with respect to $y$ and $y'$, belong to $\Omega_0$ and $\Delta_0$, respectively. 

%In this (easy) case, 
%we can use Lemma~\ref{firstclosing} twice
%to close those two $\varepsilon$-pseudo periodic orbits. Therefore we obtain  
%a $C^0$-$\varepsilon$-perturbation 
%$\widetilde g \in \widetilde{\text{Homeo}}_0(\Bbb T^2)$
%of $\widetilde f$, whose projection $g$
%admits
%two periodic orbits 
%with rotation vectors in the regions $\Delta_0$ and $\Omega_0$, respectively.\\ 

%\noindent{\bf Case Two.} For points satisfying those conditions in (\ref{epsilon_close}) and (\ref{epsilon_away}), 
%one can only see rotation vectors in the same region, either $\Delta_0$ or $\Omega_0$. 
%The other case dealt similarly, 
From now on, we assume that $\widetilde y_0$ and $n^\ast \geq 1$
satisfy that $f^{n^\ast}(y_0)$ is $\varepsilon$-close to 
$y_0$,  and 
\begin{equation}\label{choice_of_y0}
\frac 1n [\widetilde f^{n^\ast}(\widetilde y_0)-\widetilde y_0] \in \Delta_0.
\end{equation}
The other possibility (the rotation vector obtained in the above expression belonging to $\Omega_0$) is analogous.
%Choose one such point $y_0$ and integer $n_\ast$, and obtain an $\varepsilon$-pseudo periodic orbit
%$\{y_0,\cdots, y_{n_\ast}\}$, satisfying (\ref{choice_of_y0}).
The goal is to show that,
we can always find another $6\varepsilon$-pseudo periodic orbit lifting to a $6\varepsilon$-pseudo $\widetilde f$-orbit segment starting at some $\widetilde z'$ and ending at some $\widetilde z''$,
so that the rational vector 
\begin{equation}
\frac 1n (\widetilde z''-\widetilde z') \in \Omega_0.
\end{equation}

We assume the totally irrational $(\alpha,\beta)$
has norm $1$ for simplicity. 
For any $K>0$, define $\mathcal R_K\subset \Bbb T^2$ to be the set of points $x$
such that for at least $K$ choices of positive integers $n$,
%\marginpar{Here I changed the definition so that more than K returning 
%times one find deviation along (alpha,beta) larger than $K$.}
we have 
$\text{dist}_{\Bbb T^2}(f^n(x), x)  \leq \frac 1K$ and $\text{pr}_{(\alpha,\beta)} \big(\widetilde f^n (\widetilde x) - \widetilde x - n(\alpha,\beta) \big) \geq K.$ 
We claim that $\mathcal R_K$ is non-empty for any positive 
constant $K$. In fact, we can cover $\Bbb T^2$ with $N$ 
disks (for some integer $N$), all with diameter $\frac 1K.$ 
Then by the assumption on unbounded deviation  in the direction $(\alpha, \beta)$,
it is not hard to find $\widetilde x$, 
and integers $0=m_0<m_1<\cdots<m_{KN}$, 
such that, for any 
$k=1,\cdots,KN$,
\begin{equation}
\text{pr}_{(\alpha,\beta)}
 \big ( \widetilde f^{m_k}(\widetilde x) 
- \widetilde f^{m_{k-1}}(\widetilde x)  
- (m_k-m_{k-1})(\alpha,\beta) \big) \geq K.
\end{equation}
  By the pigeonhole principle, 
  for at least $K+1$ choices of the indices among those $m_j$'s, the corresponding 
  iterates of $x_0$ lie in one single disk with diameter no more than $\frac 1K$. 
  Clearly, between any two of these chosen ones, say $m_i<m_j$, we see 
    \begin{equation}
    \text{pr}_{(\alpha,\beta)} 
    \big(\widetilde f^{m_j} (\widetilde x) 
                   - \widetilde f^{m_i}(\widetilde x) 
  -(m_j-m_i)(\alpha,\beta) )  \geq K,
 \end{equation}
 and in particular the claim follows by taking the first iterate among those. 

 Next we take an accumulation point $x^\ast$ of 
the sets 
$\mathcal R_K$ as $K$ tends to infinity.  
By Lemma~\ref{dense}, the point $x^\ast -4\varepsilon (-\beta,\alpha)$ 
%(this expression makes sense when $\varepsilon$ is small) 
is approximated by an $f$-recurrent point. Then, with the help of Lemma~\ref{displacement}, one can find a point $z_1$ which is some forward iterate of the recurrent point we just found, and a positive integer $n_1$, such that both  $z_1$ and $f^{n_1} (z_1)$ are $\varepsilon$-close to $x^\ast -4\varepsilon (-\beta,\alpha)$ and 
\begin{equation}
	\text{pr}_{(-\beta,\alpha)} \big ( \widetilde f^{n_1} (\widetilde z_1)- 
 \widetilde{z}_1 \big ) < \varepsilon. 
\end{equation}

Next, by choosing another $f$-recurrent point near the point $ f^{n_1} (z_1)- 4\varepsilon (-\beta,\alpha)$ and then applying Lemma~\ref{displacement}, we can find 
another orbit segment satisfying similar estimates. 
In fact, we can inductively find 
$z_1, z_2, \cdots,z_{K_0}$ with disjoint orbits, and integers $n_1,n_2,\cdots,n_{K_0}$, such that, for any $i=2,\cdots, K_0$, both $z_i$ and $f^{n_i}(z_i)$ are $\varepsilon$-close to 
$f^{n_{i-1}}(z_{i-1})-4\varepsilon(-\beta,\alpha)$ and 
\begin{equation}
	\text{pr}_{(-\beta,\alpha)} \big ( \widetilde f^{n_i} (\widetilde z_i)- 
	\widetilde{z}_i \big ) < \varepsilon. 
\end{equation} 
Moreover, $K_0$ is chosen as the smallest integer such that 
$f^{K_0}(z_{K_0})-4\varepsilon (-\beta,\alpha)$ is $\varepsilon$-close to  $x^\ast$ in $\Bbb T^2$. 
It is not difficult to ensure 
$K_0$ is found to be finite and the process stops. 
Then we sum all the deviations (in the direction $(\alpha,\beta)$) of each of the above orbit segments.  
\begin{equation}\label{M2}
M := 
\sum_{i=1}^{K_0} \text{pr}_{(\alpha,\beta)} 
\big( 
\widetilde f^{n_i}(\widetilde z_i) -
\widetilde z_i -
n_i(\alpha,\beta) \big).
\end{equation}

%\marginpar{here I added one paragraph, 
%hope it is fine now.}
When we consider a point $x$ 
in $\mathcal R_K$ with $K> |M| +\varepsilon(K_0+1)$,
 by definition, 
for at least $K$ positive integers, its corresponding iterates all lie in a same disk of diameter at most 
$1/K$, and pairwise the deviation along the direction $(\alpha,\beta)$ is 
at least $|M| +\varepsilon(K_0+1)$. Moreover, the dynamics has bounded deviation along the perpendicular direction (see Theorem~\ref{Fabio_Bounded}). 
So when $K$ is sufficiently large, we can find two of these iterates so that along the direction $(-\beta,\alpha)$, the deviation is at most $\varepsilon$ (with a similar argument as in Lemma~\ref{displacement}).

Recall now that $x^\ast$ is an accumulation point of $\mathcal R_K$. Therefore, by the above paragraph, we can find $z_0$ 
and an integer $n_0$, 
such that both $z_0$ and $f^{n_0}(z_0)$ are $\varepsilon$-close to $x^\ast$, 

\begin{align}
& \text{pr}_{(-\beta,\alpha)} 
\big( \widetilde f^{n_0}(\widetilde z_0) - 
\widetilde z_0 \big) < \varepsilon \text{ and } \label{jump_01} \\
& \text{pr}_{(\alpha,\beta)} 
\big( \widetilde f^{n_0}(\widetilde z_0) - 
\widetilde z_0  - n_0 (\alpha,\beta)\big)
 > |M| +\varepsilon(K_0+1). \label{large_deviation}
\end{align} 

Moreover, we can
require that the orbits of $z_i$, $i=0,\cdots,K_0$ and $y_0$ are all 
pairwise disjoint. 
Then we write down the following $K_0+1$ point-wise disjoint 
  $f$-orbit segments, namely
\begin{equation}\label{K+1segments} 
  \{z_1, f(z_1), \cdots, f^{n_1}(z_1)\}, \cdots,
 \{z_{K_0}, \cdots, f^{n_{K_0}}(z_{K_0})\}, 
  \{z_0, \cdots, f^{n_0}(z_0)\}.
\end{equation}
They together form a $6\varepsilon$-pseudo periodic orbit of period 
 \begin{equation}
 \ell= \sum_{j=0}^{K_0}n_j.
 \end{equation} 
Note that, (\ref{large_deviation}) implies the final deviation of the whole pseudo-orbit 
along $(\alpha,\beta)$ is positive. 
%\marginpar{added this paragraph to explain a bit more.}
Among those $K+1$ segments in (\ref{K+1segments}), 
the way we jump between two consecutive orbit segment gives at least $2\varepsilon$ deviation along the direction $(\beta,-\alpha)$, and within each segment the deviation along $(-\beta,\alpha)$ is at most $\varepsilon$. So in the end 
we see positive deviation along $(\beta,-\alpha)$. It follows that 
this pseudo orbit sees a rotation vector in the region $\Omega_0$. 

Thus, we can apply Lemma~\ref{firstclosing} twice in order to
   close this and the $\varepsilon$-pseudo periodic orbits containing $y_0$
    which we found at the beginning of this proof.
   So we obtain $\widetilde g$ which is  $6\varepsilon$-close to $\widetilde f$ in the $C^0$ topology, and admits two periodic points $y_0$ and $z_1$.    
   By (\ref{choice_of_y0}), 
 $\rho(\widetilde g, y_0)\in \Delta_0$. And by the above construction, it follows that $\rho(\widetilde g, z_1)\in \Omega_0$. 
  
Since $(0,0)$ is an extremal point of $\rho(\widetilde f)$, 
by ~\cite{Franks_ETDS_8}
$f$ admits a contractible fixed point $p^\ast$. 
Now as we look back at
the whole perturbation process above, 
we can choose all the orbit segments 
far from $p^\ast$. 
This means that
 the perturbations can be made away from $p^\ast$. 
 Thus, the rotation set of  $\widetilde g$ satisfies 
 $\rho(\widetilde g) \supset
 \{(0,0), \rho(\widetilde g, z_1), 
 \rho(\widetilde g, y_0)\}$. Since any rotation set is convex (Theorem~\ref{rotationset}), 
 $\rho(\widetilde g)$ is as required. 
 As $\varepsilon$ can be arbitrarily small, the proof of Theorem~\ref{instability} is completed.
\end{proof}

\subsection{Origin as Endpoint}
%\marginpar{\color{red}{The proof of this section has to be changed a bit, since we are considering pseudo-orbits. I checked it seems everything goes through.}}

\begin{proof}[Proof of Theorem~\ref{case_of_0}] 
This proof is similar to the above one. % in Theorem~\ref{instability}. 
We fix $\varepsilon$ now. 
First let us define two new regions as follows.
\begin{align}
\mathcal D_0: = \{\widetilde z \in \Bbb R^2 \big| \text{pr}_{(\alpha,\beta)}(\widetilde z)<0 \text{ and } 
\text{pr}_{(-\beta,\alpha)} (\widetilde z) >0 \} .\\
\mathcal D_1: = \{\widetilde z \in \Bbb R^2 \big| \text{pr}_{(\alpha,\beta)}(\widetilde z)<0 \text{ and } 
\text{pr}_{(-\beta,\alpha)} (\widetilde z)<0 \}.
\end{align}

%\marginpar{here the definition of $\mathcal R'_K$ is changed similarly. Also the condition has changed back to the original orbit segments instead of pseudo orbits, with unbounded deviation along $-(\alpha,\beta)$}
For any $K>0$,
define $\mathcal R'_K$ to be the set of points $z$ such that for at least $K$ choices of integers $n>0$, the following holds. 
\begin{align}
              \text{dist}_{\Bbb T^2} 
                      (f^n(z), z )    &   \leq    \frac 1K. \label{recurrence} \\
                       \text{pr}_{(\alpha,\beta)}
                       (\widetilde f^n(\widetilde z)-\widetilde z)      &   
                       \leq  -K. \label{deviation}
\end{align}
    With the condition of unbounded deviations in the direction $-(\alpha,\beta)$, 
    by a similar argument as in previous subsection, we can show $\mathcal R'_K$ is non-empty for any $K>0$. 
     
   % There are two cases to consider.
%   In the first case, 
 %  we can find two 
 % We can then find an $\varepsilon$-pseudo periodic point, namely $y$ and $y'$ with disjoint orbits,
  % and integers $n,n'$, so that $f^n(y)$ is $\varepsilon$-close to $y$, $f^{n'}(y')$ is 
  % $\varepsilon$-close to $y'$, and that 
   % \begin{equation}
    % v= \frac{1}{n}[\widetilde f^{n}(\widetilde y) -\widetilde{y}] \in \mathcal D_0, \text{ and }
   % v'= \frac{1}{n'}[\widetilde f^{n'}(\widetilde y') -\widetilde{y}'] \in \mathcal D_1.
    % \end{equation}
%It means that we obtain two $\varepsilon$-pseudo orbits, seeing rotation vector in $\Delta_0$ and $\Omega_0$, respectively.  
 
 By taking $K\geq \frac 1\varepsilon$, 
 we can find an $\varepsilon$-pseudo periodic point $y_0$ of period $n^\ast$, 
 which realizes a rotation vector in the region $\mathcal D_0\cup \mathcal D_1$. 
 Without loss of generality, assume it lies in $\mathcal D_0$ (the other possibility is
analogous). 
 More precisely, for a lift $\widetilde y_0$ of $y_0$ and some integer $n_\ast>0$, 
 $\text{dist}_{\Bbb T^2}(f^{n_\ast}(y_0)-y_0) <\varepsilon$, and 
    \begin{equation}
    v =\frac{1}{n_\ast}[\widetilde f^{n_\ast}(\widetilde y_0)- \widetilde y_0]  \in \mathcal D_0.
    \end{equation}
 
    Then, very similarly to the previous subsection, let $y^{\ast}$ be an accumulation point of 
   $\mathcal R'_K$ as $K$ tends to infinity. With the help of Lemmas~\ref{dense} 
and~Lemma~\ref{displacement}, as well as the definition of $\mathcal R_K'$,
    we can choose finitely many orbit segments, 
    which altogether form a $6\varepsilon$-pseudo periodic orbit, lifting to a $6\varepsilon$-pseudo orbit for $\widetilde f$,
    namely $\{\widetilde z_0, \widetilde z_1, \cdots, \widetilde z_{n'}\}$, such that 
    \begin{equation}
   v'= \frac{1}{n'}[\widetilde z_{n'} -\widetilde z_0] \in \mathcal D_1.
          \end{equation} 
   
 Thus, we obtain two pseudo periodic orbits seeing rotation vectors 
 $v\in \mathcal D_0$ and $v' \in \mathcal D_1.$ 
  Since $(\alpha,\beta)$ is an 
  extremal point of $\rho(\widetilde f)$, 
 by Theorem~\ref{rotationset}, there exists some ergodic $f$-invariant 
measure $\mu$
 satisfying $\rho_\mu(\widetilde  f)=(\alpha,\beta)$.
 Moreover, for a $\mu$-typical point $x \in \Bbb T^2$, 
 for some increasing integer sequence $n_j$, 
 and any lift $\widetilde {x}$, 
 \begin{align}  
 \lim_{j\to \infty} f^{n_j}(x)  & = x.\\
 \lim_{j\to \infty}
 \frac{1}{n_j}  \big( \widetilde f^{n_j}(\widetilde x) -\widetilde x \big )& = (\alpha,\beta).
 \end{align}
Then for sufficiently large $n_j$, 
 $\{x, f(x),\cdots, f^{n_j}(x)\}$ forms an $\varepsilon$-pseudo periodic 
 orbit, such that the vectors $v, v'$ and 
 \begin{equation}
 w = \frac {1}{n_j} [ \widetilde f^{n_j}(\widetilde x) -\widetilde x ]
 \end{equation}
 spans a triangle, which contains the origin $(0,0)$ in its interior. 
 Note that we can choose the three pseudo orbits to be pair-wise disjoint. 
 Then, applying Lemma~\ref{firstclosing} three times, we
obtain a $6\varepsilon$-perturbation $\widetilde g$ of $\widetilde f$. 
The rotation set 
$\rho(\widetilde g)$ contains
 at least the three rational points 
 $v,v'$ and $w$. 
 By the convexity of the rotation set, 
 $(0,0)$ is contained in  $\text{Int}\big(\rho(\widetilde g) \big)$. 
 We have completed the proof.
\end{proof}

\section{Perturbations for Homeomorphisms with Bounded Deviation}

\subsection{The Totally Irrational Rigid Rotation}\label{proofs_bounded}\hfill\\

 We start by showing a simple perturbation result for the rigid rotation $R_{(\alpha,\beta)}$ on $\Bbb T^2$, 
where $(\alpha,\beta)$ is totally irrational.
\begin{prop}\label{Rigid_Case} 
	For any $\varepsilon>0$, there exists 
$\widetilde g\in \widetilde{ \text{Homeo}}_0(\Bbb T^2)$ 
with $\text{dist}_{C^0}(\widetilde g,R_{(\alpha,\beta)})<\varepsilon$,
such that, $\rho(\widetilde g)$ has interior, 
and $(\alpha,\beta) \in \text{Int}(\rho(\widetilde g))$.
\end{prop}
\begin{proof} 
Since $(\alpha,\beta)$ is totally irrational, 
the rigid rotation $R_{(\alpha,\beta)}$ is minimal. 
For any $x_0\in \Bbb T^2$ and for any small $\varepsilon>0$, 
consider a small disk 
$B=B(x_0,\varepsilon)$, which is divided
into four regions as follows.
\begin{align}
\Delta_1(x_0,\varepsilon):      
& = \{x\in B(x_0,\varepsilon) 
\big| \text{pr}_{(\alpha,\beta)}(x-x_0)< 0, \text{pr}_{(-\beta,\alpha)}(x-x_0)<0\}. \label{new_delta_1}\\
\Delta_0 (x_0,\varepsilon):      
& = \{x\in B(x_0,\varepsilon) 
\big| \text{pr}_{(\alpha,\beta)}(x-x_0)> 0, \text{pr}_{(-\beta,\alpha)}(x-x_0)> 0\}.\\
\Omega_1(x_0,\varepsilon):   
& =  \{x\in B(x_0,\varepsilon) 
\big| \text{pr}_{(\alpha,\beta)}(x-x_0) < 0, \text{pr}_{(-\beta,\alpha)}(x-x_0)> 0\}.\\
 \Omega_0(x_0,\varepsilon):  
 & = \{x\in B(x_0,\varepsilon) 
 \big| \text{pr}_{(\alpha,\beta)}(x-x_0)> 0, \text{pr}_{(-\beta,\alpha)}(x-x_0)<0\}. \label{new_omega_0}
\end{align}
By minimality of  $R_{(\alpha,\beta)}$,
we can choose some integer $n$,
such that $R_{(\alpha,\beta)}^{n}(x_0) \in \Delta_1(x_0,\varepsilon)$. 
Then, for proper choices of 
lifts $\widetilde R$ and $\widetilde{x_0}$
 of $R_{(\alpha,\beta)}$ and $x_0$, respectively, 
 $\widetilde R^{n}(\widetilde x_0)$ is $\varepsilon$-close to 
$\widetilde x_0 +(a,b)$ for some $(a,b)\in \Bbb Z^2$. 
Moreover, we can write $v_1=(\frac{a}{n},\frac{b}{n})$,
and then clearly $v_1 \in \Delta_0((\alpha,\beta),\frac{\varepsilon}{n})$.

In other words, 
we can find an $\varepsilon$-pseudo periodic 
orbit for the rigid rotation $R_{(\alpha,\beta)}$,
which sees a rational rotation vector in the region 
$\Delta_0((\alpha,\beta),\frac{\varepsilon}{n})$.
We argue in a similar way with respect to 
the other three regions. Then, we obtain 
four $\varepsilon$-pseudo periodic orbits for $R_{(\alpha,\beta)}$,
starting with $x_0, y_0,z_0,w_0$, respectively. 
We can also require these orbit segments to be point-wise disjoint. Then, 
applying Lemma~\ref{firstclosing} four times,
these four pseudo orbits can be closed via 
an $\varepsilon$-perturbation, 
which produces four periodic orbits.  
These periodic orbits will 
have four rational rotation vectors $v_1,v_2,v_3,v_4$, respectively, 
whose convex hull contains $(\alpha,\beta)$
in the interior. Therefore $(\alpha,\beta)\in \text{Int}(\rho(\widetilde g))$. 
\end{proof}

\begin{Remark} 
This proposition can be compared with 
Theorem 1 in~\cite{Karaliolios}, where 
with respect to sufficiently  
high regularity, due to the 
``KAM" phenomenon,
the perturbed rotation set either misses 
$(\alpha,\beta)$, or it equals $\{(\alpha, \beta)\}$, provided 
$(\alpha,\beta)$ satisfies certain Diophantine conditions.
%Still for this 
%rigid rotation $R_{(\alpha,\beta)}$, 
%it is unknown if one could
%make arbitrarily small 
%$C^1$-perturbation,
% to obtain rotation set whose interior contains $(\alpha,\beta)$.
\end{Remark}

\subsection{Bounded Deviations}\hfill\\

In this subsection, we assume that $\widetilde f$ has bounded deviation 
along the direction $(\alpha, \beta)$. We consider this case for the sake of completeness, but 
it is possible that it might not happen at all (cf. Theorem~\ref{generic_bounded_unbounded} and Question~\ref{Question_unbounded}). 
 \begin{theorem}\label{Bounded_Deviation_Case}
 Suppose $\widetilde f\in \widetilde{\text{Homeo}}_0(\Bbb T^2)$, whose rotation set
 $\rho(\widetilde f)$ is the segment from $(0,0)$ to the totally irrational point $(\alpha,\beta)$. 
 Assume $\widetilde f$ has bounded deviation along the direction $(\alpha,\beta)$. Then $\widetilde f$ can be $C^0$-approximated by $\widetilde g\in  \widetilde{\text{Homeo}}_{0}(\Bbb T^2)$,
 such that 
 $\rho(\widetilde g)$ has interior,
  and 
 $\rho(\widetilde f)\backslash\{(0,0)\} \subset \text{Int} (\rho(\widetilde g) )$.
 \end{theorem}
 
 Assume for some $M>0$, for any $\widetilde x \in \Bbb R^2$
  and any $n\geq 1$,
 \begin{equation}\label{Bounded_by_M}
 \text{pr}_{(\alpha,\beta)}( \widetilde f^n (\widetilde x) -\widetilde x -n(\alpha,\beta)) \leq M. 
 \end{equation}
   %Let us first 
   %use Atkinson's Lemma to obtain certain
   %consequence of bounded deviation.
   %In the following lemma, we give some estimates. 
   %Similar considerations were observed in several %works in similar contexts. %$(see~\cite{Salvador_Uniform} and ~\cite{Forcing}.)
   \begin{definition}\label{Mab_Sab}
   Let $\mathcal M_{(\alpha,\beta)}$ 
   denote the set of ergodic $f$-invariant Borel probability measures, 
   which have $(\alpha,\beta)$ as rotation vector.% $\rho_\mu(\widetilde f)$ 
   %are equal $(\alpha,\beta)$.%such that  
   %$\mu \in \mathcal M_{(\alpha,\beta)}$ if and only if 
   %$\rho_\mu(\widetilde f)=(\alpha,\beta)$.
   Then write 
   \begin{equation}
   \mathcal S_{(\alpha,\beta)} :=
   \overline {\bigcup_{\mu \in \mathcal M_{(\alpha,\beta)}} \text{supp}(\mu)},
   \end{equation} 
   where $\text{supp}(\mu)$ denotes the support of $\mu$.
\end{definition}
The following Lemma, %is classic now, 
whose proof depends on Atkinson's theorem on Cocycles (see~\cite{Atkinson}),
appears as Lemma 6 of~\cite{Salvador_Uniform} or Proposition 65 of~\cite{Forcing}. 
\begin{lemma}\label{geq_-M}
   Suppose $\widetilde f$ satisfies condition (\ref{Bounded_by_M}).
    Then for any $x\in \mathcal S_{(\alpha,\beta)}$ with a lift $\widetilde x$,
    and for any $n\geq 1$,
   \begin{equation}\label{Bounded_Below}
   \text{pr}_{(\alpha,\beta)} \big(\widetilde f^n(\widetilde x) -\widetilde x -n(\alpha,\beta) \big) \geq -M.
      \end{equation}
In particular, 
     any invariant ergodic measure $\mu$ such that $\text{supp}(\mu) \subset \mathcal S_{(\alpha,\beta)}$ 
    is contained in $\mathcal M_{(\alpha,\beta)}$. 
   \end{lemma}

        \begin{proof}[Proof of Theorem~\ref{Bounded_Deviation_Case}]
        Fix $\varepsilon>0$. Choose any minimal set $K\subset \mathcal S_{(a,b)}$.
        	Then $\rho(\widetilde f, K)=\{(\alpha,\beta)\}$. 
        	By assumption (\ref{Bounded_by_M}), 
        	Theorem~\ref{Fabio_Bounded} and Lemma~\ref{geq_-M}, 
        	$\widetilde f\big |_{\pi^{-1}(K)}$ 
        	has bounded deviation along every direction. 
        	So we can apply Lemma~\ref{jager} in order to 
        	find a semi-conjugacy $\phi$
        	between $(K,f\big|_K)$ and 
        	$(\mathbb T^2,R)$, where $R= R_{(\alpha,\beta)}$ 
        	denotes the rigid rotation on $\mathbb T^2$ by $(\alpha,\beta)$.
        	%Moreover, $\phi$ can chosen to be homotopic to the identity.
         There is a lift $\widetilde \phi$ of $\phi$,
        	which conjugates 
        	$\widetilde f\big|_{\pi^{-1}(K)}$ and $\widetilde R$.
        	%Since there exists a 
        	%homotopy between the identity 
        	%and $f$, the homotopy lifts to the coverings. 
        	%It implies 
	Note that the pre-image of every point under $\widetilde \phi$
        	has diameter uniformly bounded from above, say by a constant $C_\phi>0$.
	Up to renormalizing the torus to a finite cover, we can assume that 
	\begin{equation}\label{bounded_diameter}
	C_\phi<1/6.
	\end{equation}	
	 Then, there exists a positive integer $n_0$, such that for any $n\geq n_0$, the following can be ensured. For any $x\in\Bbb T^2$ and its pre-image $\gamma=\phi^{-1}(x)$, 
	 and for any points $p\in \gamma$ and $q\in f^n(\gamma)$, 
	 one can find an $\varepsilon$-pseudo orbit segment of length $n$, 
	starting at $p$ and ending at $q$, such that every jump happens within a leaf $\phi^{-1}(f^k(\gamma))$, for $k\geq 1$. 
	
	Recall the four regions defined  from (\ref{new_delta_1}) to 
        	(\ref{new_omega_0}).
        	    	Since $R= R_{(\alpha,\beta)}$ 
        	is minimal, for some sufficiently large positive integer $n > n_0$, 
	and for a point $x_0\in \Bbb T^2$ with its pre-image $\gamma_0=\phi^{-1}(x_0)$, 
	we have that  
        	\begin{align}
        		 & R^{n}(x)\in \Omega_1(x,\varepsilon). \label{belong_to_omega_1}\\
		 &  \text{dist}_{\Bbb T^2}(f^n(\gamma_0), \gamma_0)<\varepsilon. \label{gamma_0_close}
			    	\end{align}
           Estimate (\ref{belong_to_omega_1}) means that, by choosing a lift $\widetilde x_0$ of $x_0$, and writing $(a,b)= [\widetilde R^{n}(\widetilde x_0) - \widetilde x_0]$, we have  
        	\begin{equation}\label{vk_definition}
        		v= \frac{[\widetilde R^{n}(\widetilde x_0) - 
        			\widetilde x_0]}{n}= (\frac{a}{n}, \frac{b}{n}) \in \Omega_0.
        	\end{equation}
        	Equivalently (see Figure~\ref{the_four_regions}),
        	\begin{align}
        		\text{pr}_{(\alpha,\beta)} \big (v-(\alpha,\beta) \big)   > 0. \label{pr_alphabeta}\\
        		\text{pr}_{(-\beta,\alpha)} \big (v-(\alpha,\beta) \big)  < 0. \label{pr_betaalpha}
        	\end{align}          

On the other hand, by estimate (\ref{gamma_0_close}) and noting $n> n_0$, 
           we can find an $\varepsilon$-pseudo periodic orbit, containing a point 
           $p\in\gamma_0$ so that for some $q\in f^n(\gamma_0)$, 
           \begin{equation}
           \text{dist}_{\Bbb T^2}(p, q)= 
           \text{dist}_{\Bbb T^2}(\gamma_0,f^n(\gamma_0))<\varepsilon.
           \end{equation} 
           This pseudo-orbit is obtained in the following way. 
           For each $k=1,\cdots,n-1$, the orbit has a jump within 
           the pre-image $\phi^{-1}(f^k(\gamma_0))$, so that 
           at the $(n-1)$-th time, it arrives at $f^{-1}(q)$. Then the final jump
           happens from $q$ to $p$. 
                   This pseudo-orbit lifts to 
                   an $\varepsilon$-pseudo orbit for $\widetilde f$ in $\Bbb R^2$. 
                   Recall from (\ref{bounded_diameter}) that, 
                   the pre-image under the lifted semi-conjugacy $\widetilde \phi$ of any point
                   has diameter bounded by $1/6$. It follows that the pseudo-orbit must see the same 
                   integer translate as the rigid rotation. More precisely,         
                   the $\varepsilon$-pseudo-orbit above must start at some 
                   point $\widetilde p$ and end at $\widetilde p+(a,b),$ so that 
                   it sees the rotation vector $v=(\frac an,\frac bn)  \in \Omega_0$. 
	
          In a similar way, we can find another $\varepsilon$-pseudo $f$-periodic orbit  which sees 
	 a rational rotation vector in $\Delta_0$.
	 Moreover, as we explained before, there is also at least one contractible fixed point $p_\ast$.
         Note that we can choose the two pseudo orbits and the fixed point to be pairwise disjoint. 
			So, if we apply Lemma~\ref{firstclosing} twice, we obtain 
        	an $\varepsilon$-perturbation 
        	$\widetilde g\in 
        	\widetilde{\text{Homeo}}_0
        	(\mathbb T^2)$,
        	with at least three periodic orbits, 
        	$p_\ast$,$p_0$, $u_0$, 
        	such that,
        	$\rho(\widetilde g,p_\ast)=0, 
        	\rho(\widetilde g,p_0)=v$ and $\rho(\widetilde g,u_0)=u$
        	span a triangle which  contains $\rho(\widetilde f)\backslash \{(0,0)\}$
        	in its interior. By Theorem~\ref{rotationset},
	this triangle is clearly included in $\rho(\widetilde g)$, 
	and the proof of the Theorem is completed.
        \end{proof}
        
        \begin{proof}[Proof of Theorem~\ref{instability_general}]
        	Theorem~\ref{instability_general} follows immediately from 
        	Theorem~\ref{instability} and Theorem~\ref{Bounded_Deviation_Case}.
        \end{proof}

\section{Generic Diffeomorphisms}

In this section, we prove the following theorem.

\begin{theorem}
\label{generic_not_this_set}[Theorem~\ref{generic_case} restated] Let $f\in 
\mathcal G^r$. Then for any lift $\widetilde{f}$, the rotation set $\rho ( 
\widetilde{f})$ can not be a segment from $(0,0)$ to a totally irrational
point $(\alpha ,\beta )$.
\end{theorem}

{\it Proof. }The proof of this theorem will go through the whole section. We
assume the following conditions and arrive at a contradiction in the end of
the proof. 
\begin{align}\label{Gr_condition} 
&  f\in \mathcal G^r.\\
&  \rho(\widetilde f) \text{ is the segment from} (0,0) \text{ to the totally irrational point } (\alpha,\beta). \label{rotation_set_assumption}
\end{align}

%\begin{proof}[ of Theorem~\ref{generic_not_this_set}]\renewcommand%
%{\qedsymbol}{}
Since $(\alpha,\beta)$ is an extremal point of $\rho(\widetilde f)$, by
Theorem~\ref{rotationset}, we can choose an $f$-recurrent point $z_\ast$,
such that 
\begin{equation}
\label{choice_of_zast} \rho (\widetilde f,z_\ast)=\lim_{n\to \infty}\frac1n 
\big(\widetilde f^n(\widetilde{z}_\ast ) -\widetilde z_\ast \big)=
(\alpha,\beta). 
\end{equation}

Now, consider the family of all the $f$-invariant open topological disks. We
can define a partial order among this family with respect to the usual inclusion
relation. It is standard to check that with respect to this order, the
family forms a partially ordered set, for which every chain has an upper
bound. So by Zorn's lemma, we conclude the existence of maximal elements. As 
$f$ is non-wandering, Lemma~\ref{Bounded Invariant Disks} implies that there
exists a constant $M>0$, such that, every connected component $\widetilde{D}$
of the lift of a maximal open $f$-invariant disk $D$ is $\widetilde{f}$%
-invariant, and 
\begin{equation}
\text{diam}(\widetilde{D})<M. 
\end{equation}
Also, if $D$ is an $f$-invariant maximal open topological disk, then it
contains fixed points. This follows from a classical argument: pick some
point $p\in D$. If $p$ is not fixed, then for a sufficiently small open ball 
$B$ centered at $p$, contained in $D$, we have: $B$ is disjoint from $f(B)$
and $f^n(B)$ intersects $B$ for a sufficiently large $n>0$ (there are no
wandering points). And this implies the existence of a fixed point inside $D,$
see Lemma \ref{Brouwer}. Since for $f\in \mathcal G^r$, there are finitely
many fixed points, it follows that there are at most finitely many
maximal open $f$-invariant disks.

Note that $z_{*}$ is disjoint from the closure of the union of these
finitely many maximal open $f$-invariant disks, because the orbit of any $
\widetilde{z}_{*}\in \pi ^{-1}(z_{*})$ is unbounded in $\Bbb R^2$. Thus, we
can choose some $\delta >0$ such that, the open disk $B(z_{*},\delta )$ is
still disjoint from the closure of the union of these maximal open $f$%
-invariant disks. Define the set 
\begin{equation}
\label{defineu}
U:=\text{ the connected component of }\bigcup_{n\in \Bbb Z}f^n\big(%
B(z_{*},\delta )\big)
\text{ which contains }z_{*}. 
\end{equation}
%\end{proof}

The following lemma allows us to find a good hyperbolic saddle fixed point.

\begin{lemma}
\label{Saddle_Essential} There exists at least one fixed hyperbolic
saddle point $Q_{*}$ which is contained in $\overline{U}$.
\end{lemma}

\begin{proof}

Clearly $U$ is open. Since $f$ is non-wandering, $U$ is 
$f^{n^\ast}$-invariant for some $n^\ast>0$. Recall the notions 
in subsection~\ref{notation1}, 
and claim that $U$ is essential. 
Suppose otherwise and let $U_{filled}$ be    
the union of $U$ with all the connected components of the
 complement of $U$ 
 which are contractible. This construction implies that $U_{filled}$ 
is an open disk and, by Lemma~\ref{Bounded Invariant Disks} applied to $f^{n^\ast}$,
  all connected components of the lift of $U_{filled}$ to the plane 
have bounded diameter and are $\widetilde{f}^{n^\ast}$-invariant 
(because $(0,0)$ is the only rational point contained in the rotation set).
 In particular, any lift of $z_\ast$ has bounded orbit, 
a contradiction with (\ref{choice_of_zast}). 

The next claim is that $U$ is in fact fully essential. 
Suppose it is not, then all the homotopically 
non-trivial loops contained in $U$ are homotopic 
to each other.  Fix one homotopically non-trivial 
loop $\gamma \subset U$ and choose connected components of 
their lifts, $\widetilde \gamma$ and $\widetilde U$, such that, 
\begin{equation}
\widetilde \gamma \subset \widetilde U.
\end{equation}
Clearly, there exists some integer vector $(a,b)\neq (0,0)$, such that $\widetilde \gamma=\widetilde \gamma+(a,b)$. 
Moreover, 
since by assumption $U$ is not fully essential, 
$\widetilde U \bigcap (\widetilde U+ i(-b,a)) = \emptyset$
for any integer $i\neq 0$. Therefore, 
\begin{equation}
\widetilde U \text{ is contained in the strip bounded by } \widetilde \gamma-(-b,a) \text{ and }\widetilde \gamma+(-b,a).
\end{equation}
This is a contradiction with the choice of $z_\ast$
 and the fact that $(\alpha, \beta)$ is totally irrational. 
 So $U$ is fully essential. Moreover, if $\overline{U}$ is not the whole torus, 
then any connected component
of $(\overline{U})^c$ is a periodic open disk (the periodicity follows from the 
non-wandering hypothesis). In particular, by the choice of $\delta>0,$ 
each $f$-invariant maximal
open disk (if any) is a connected component of $(\overline{U})^c$.

By assumption~(\ref{Gr_condition}),
every fixed point has non-zero topological index.
In this case, the fixed point 
is also called a non-degenerate fixed point. 
In general there are two types of such points:
\begin{enumerate}
\item $p\in\text{Fix}(f)$ has topological index $1$.
\item $p\in \text{Fix}(f)$ 
has topological index $-1$. In this case, $p$ 
is a hyperbolic saddle, and 
both eigenvalues of $Df(p)$ are positive real numbers, one larger than $1$ and the other smaller.
\end{enumerate}

Now by Lemma~\ref{primeends}, 
condition (2) of Definition~\ref{Gr_definition}
implies that
each $f$-invariant open disk $D$ 
has prime ends rotation number $\rho_{\text{PE}}(f,D)\notin \Bbb Q$. In particular,
such a prime ends rotation number is not zero, so the sum of the indices of 
fixed points contained in 
the union of all the (finitely many) maximal $f$-invariant open disks is
positive (or zero, in case $\overline{U}$ is the whole torus). 
By the Lefschetz fixed point 
formula (Lemma~\ref{Lefschetz}),  
the sum of the indices at all the  
fixed points is zero.
And as $(0,0)$ is an extremal point of the rotation set, 
$f$ must have fixed points (see ~\cite{Franks_ETDS_8}).  
So it follows that 
there exists at least one negatively indexed fixed point, denoted $Q_\ast,$ 
contained in the complement 
of the union of these maximal open $f$-invariant disks. 
Thus, $Q_\ast$ is a fixed hyperbolic saddle point, 
which belongs to $\overline{U},$ and we have finished the proof.
\end{proof}

For a fixed hyperbolic saddle point $Q_{*}$ (or a fixed
saddle-like point), let an \textit{unstable branch} (respectively, 
\textit{stable branch}) at $Q_{*}$ be one of the connected components of $%
W^u(Q_{*})\backslash \{Q_{*}\}$ (respectively, one of the connected
components of $W^s(Q_{*})\backslash \{Q_{*}\}$). Choose a lift $\widetilde{Q}%
_{*}$ of the hyperbolic saddle point $Q_{*},$ which is fixed by $\widetilde{f%
}.$ We can then lift the corresponding stable and unstable branches at $%
Q_{*} $ to those branches at $\widetilde{Q}_{*}$.

\begin{prop}\label{stable_unstable} 
It is not possible that some unstable branch $\widetilde{\lambda}_u$
and some stable branch $\widetilde{\lambda}_s$ at the hyperbolic saddle 
$\widetilde {Q}_\ast$ intersect.
\end{prop}

\begin{proof} 
Suppose by contradiction that a stable branch $\widetilde{\lambda}_s$
at $\widetilde {Q}_\ast$ intersects 
an unstable branch $\widetilde{\lambda}_u$ at $\widetilde{Q}_\ast$. We 
can then choose an intersection point 
$\widetilde w$, 
such that, 
the arc along $\widetilde{\lambda}_u$ from $\widetilde{Q}_\ast$ to $\widetilde w$
and the arc along $\widetilde{\lambda}_s$ from $\widetilde {Q}_\ast$ to 
$\widetilde w$ are disjoint, except at their endpoints.
It follows that, the union of these two arcs bounds a topological disk $\widetilde D$. 
Now we 
define 
\begin{equation}
\widetilde D_{\text{sat}} =\bigcup_{n\in \Bbb Z} \widetilde f^n(\widetilde D).
\end{equation}
Note that $\widetilde D_{\text{sat}}$ is an open and connected 
%and bounded 
$\widetilde f$-invariant  subset
of the plane. 

If there exists some integer vector  
$(a,b)\in \Bbb Z^2\backslash \{(0,0)\}$,
such that 
\begin{equation}
\widetilde D_{\text{sat}} \bigcap \big(\widetilde D_{\text{sat}}+(a,b) \big) \neq \emptyset,
\end{equation}
then either $\widetilde{\lambda}_u$ intersects 
$\widetilde{\lambda}_s+(a,b)$ topologically transversely, 
or $\widetilde{\lambda}_u$ intersects 
$\widetilde{\lambda}_s -(a,b)$ topologically transversely (see definition~\ref{topological_transverse}).
In both cases, it follows from Lemma~\ref{create_topological_horseshoe} that there exists a 
periodic orbit $p_0$ whose rotation vector $\rho(\widetilde f,p_0)$ 
is non-zero and rational, which is a contradiction with 
assumption (\ref{rotation_set_assumption}).

Thus, we are left with the case when 
$\widetilde D_{\text{sat}}$ does not
intersect any of its non-trivial integer translations. 
As before, we consider the filled open set 
$\text{Fill}(\widetilde D_{\text{sat}}),$ which is given by  
the union of $\widetilde D_{\text{sat}}$ and 
all the bounded connected components of its complement.
%the complement of $\widetilde D_{\text{sat}}$. 
It is not hard to see that 
$\text{Fill}(\widetilde D_{\text{sat}})$ 
is an open topological disk,
which does not intersect any of its non-zero integer translations. 
Thus, we can consider its projection 
$D_{\text{Fill}}:= \pi (\text{Fill}(\widetilde D_{\text{sat}}))$, which is an 
$f$-invariant open disk. By Lemma~\ref{Bounded Invariant Disks}, 
$\text{Fill}(\widetilde D_{\text{sat}})$ has bounded diameter. 
Since $f\in \mathcal G^r$, in particular there are no saddle connections, 
by Lemma~\ref{primeends}, 
the prime ends rotation number 
$\rho_{\text{PE}}(\widetilde f, \text{Fill}(\widetilde D_{\text{sat}}) )$ 
is irrational, and so
 the boundary 
$\partial \text{Fill}(\widetilde D_{\text{sat}})$ 
does not contain any periodic point. 
This implies that 
\begin{equation}
\widetilde {Q}_\ast \in \text{Fill}(\widetilde D_{\text{sat}}).
\end{equation} 
This is a contradiction  
because  $Q_\ast$ does not to belong to 
a fixed open disk, see Lemma~\ref{Saddle_Essential}.
So $\widetilde \lambda_u$ does not 
intersect any stable branch at $\widetilde {Q}_\ast$.
\end{proof}

\begin{Remark}\label{no_homoclinic_intersection_saddle_like}
  The same conclusion is also true for the unstable and the stable branch
  at an index $0$ 
saddle-like fixed point. The proof follows the 
same lines as above, with a difference that, 
when we apply Lemma~\ref{create_topological_horseshoe} in the arguments, we actually need the 
statement for saddle-like fixed points, as was explained in Remark~\ref{saddle_like_points_horseshoe}.
Note also, the conditions stated in Lemma~\ref{primeends} are
such that in both cases it can be applied. 
\end{Remark}

\begin{lemma}
\label{stable_unbounded} Each stable or unstable branch at $\widetilde{Q}%
_{*} $ is unbounded in $\Bbb R^2$.
\end{lemma}

\begin{proof}
For definiteness, 
fix any unstable branch 
$\widetilde{\lambda}_u$ at $\widetilde{Q}_\ast$, 
and assume by contradiction 
that it is bounded. 

The first claim is that, the closure $\text{cl} (\widetilde{\lambda}_u)$ must intersect 
all the other branches at $\widetilde{Q}_\ast$. 
To see the claim, assume by contradiction that
$\text{cl} ( \widetilde{\lambda}_u)$ 
does not intersect
some branch $\widetilde \lambda$. 
Then there exists some 
connected component $\widetilde U$ 
of the 
complement of 
$\text{cl}(\widetilde {\lambda}_u)$,
containing $\widetilde \lambda$. 
Since $\widetilde \lambda$ is $\widetilde f$-invariant,
so is $\widetilde U$.
%If we look at the one point compactification of the plane, 
%which is homeomorphic to $S^2$, 
%the union of the region $\widetilde U$ with the point $\infty$ can be regarded 
%as a simply connected open 
%$\widetilde f$-invariant subset of $S^2$.
Note that  $\widetilde{Q}_\ast \in \partial \widetilde U$, and 
it is in fact accessible through the branch
$\widetilde \lambda$, from the interior of $\widetilde U$. 
Thus, the prime ends rotation number 
$\rho_{\text{PE}}(\widetilde f,\widetilde U)$ must be equal to $0$.
And so, Lemma \ref{primeends} implies the existence of saddle-connections, 
something that contradicts item (2) of Definition~\ref{Gr_definition}.

The second claim is that, if $\widetilde \lambda$ 
is any other branch at $\widetilde{Q}_\ast$, 
then $\text{cl}(\widetilde{\lambda}_u)\supset \widetilde \lambda$.
To prove this claim, note first that if $\widetilde \lambda$ is 
another unstable branch, then $\widetilde \lambda$ does not intersect 
$\widetilde{\lambda}_u$. 
And if $\widetilde \lambda$ is a stable 
branch, then by Proposition~\ref{stable_unstable}, 
we again obtain that $\widetilde \lambda$ does not intersect 
$\widetilde{\lambda}_u$. 
The following argument 
is a variation of one due to Fernando Oliveira in the area-preserving case
(see Lemma 2 of  ~\cite{Oliveira}).
We include it here for completeness.

Assume by contradiction that 
\begin{equation}\label{not_subset}
\text{cl}(\widetilde{\lambda}_u) \not\supset \widetilde \lambda.
\end{equation}
Since $\text{cl}(\widetilde{\lambda}_u)$ is a connected $\widetilde f$-invariant 
compact subset, 
there is a compact simple arc $\gamma$ 
contained in $\widetilde \lambda$, such that $\text{cl}(\widetilde {\lambda}_u)\bigcap \gamma$
consists of exactly the two endpoints of $\gamma$. Then
 there are two possibilities: 
 
\begin{enumerate}
\item for all non-zero integer vectors $(m,n)$, 
$\gamma\bigcap  \big (\text{cl}(\widetilde {\lambda}_u)+(m,n) \big)=\emptyset$.
\item for some $(m_0,n_0) \in \Bbb Z^2\backslash \{(0,0)\}$, 
$\gamma\bigcap  \big( \text{cl} (\widetilde {\lambda}_u)+(m_0,n_0) \big) \neq \emptyset$.
\end{enumerate}

In case (1), %by assumption (\ref{not_subset}), 
we can find a bounded connected component 
of the complement of $\text{cl}(\widetilde{\lambda}_u)\cup \gamma$, whose 
boundary contains $\gamma$. 
We denote it as $\widetilde O$. Then look at the set $O=\pi(\widetilde O)$. 
It is not hard to see that, $O$ contains 
a wandering domain for $f$, 
which is a contradiction.

In case (2), as $\gamma$ is contained in a branch at $\widetilde{Q}_\ast,$ 
$\text{cl} (\widetilde {\lambda}_u)$ intersects 
$\text{cl} (\widetilde {\lambda}_u)+(m_0,n_0)$. So we define
%let us denote $\gamma'$ 
%a compact arc starting at 
%$\widetilde{Q}_\ast$ along
%the stable branch $\widetilde \lambda$, which 
%contains $\gamma$
%as a sub-arc. Now we define 
\begin{equation}
\widetilde L: =
\bigcup_{k\in \Bbb Z} \text{cl}(\widetilde {\lambda}_u ) + k(m_0,n_0).
\end{equation}
Then $\widetilde L$ is closed, connected,  
$\widetilde L + (m_0,n_0) =\widetilde L$ and it is bounded in the direction 
perpendicular to $(m_0,n_0).$
Moreover, $\widetilde f(\widetilde L)=\widetilde L.$
%(note that in case  
%$\widetilde {\lambda}$ is the other unstable branch, 
%we have instead  
%$\widetilde f^{-1}(\widetilde L)\subset \widetilde L$).
But this shows that $\rho(\widetilde f)$ must be contained in 
a line of rational slope (parallel to the vector $(m_0,n_0)$), which is a contradiction with 
assumption
(\ref{rotation_set_assumption}). This shows our second claim. 

Note that we could have started with any stable branch $\widetilde{\lambda_s}$ as well. So 
these two claims above show that, under the assumptions 
(\ref{Gr_condition}) and (\ref{rotation_set_assumption}),
if one branch at the hyperbolic fixed point 
$\widetilde {Q_\ast}$ is bounded, then all four branches are bounded. 
Moreover, in this case, they all have the same closure. 
Then the final arguments follow exactly 
from Oliveira in \cite{Oliveira}, page 582, namely we get 
an intersection between a stable and an unstable branch.
As this homoclinic intersection contradicts Proposition~\ref{stable_unstable},  
all four branches at $\widetilde{Q}^{*}$ are unbounded.

To conclude, we present a sketch of Oliveira's 
argument used above. Let us denote the four branches at $\widetilde{Q}_{*}$ 
by $\widetilde{\lambda }_{1,u},\widetilde{\lambda }_{2,u},\widetilde{\lambda }_{1,s}$ 
and $\widetilde{\lambda }_{2,s}$. We are assuming that cl$(\widetilde{\lambda }_{i,j})$ is 
bounded and equal to the same continuum for all $i\in \{1,2\}$ 
and $j\in \{u,s\}.$ 
This implies that there exist two branches, one stable and one unstable, and a 
local quadrant $Quad$ at $\widetilde{Q}^{*}$ adjacent to at least one of them, 
such that both branches accumulate on $\widetilde{Q}^{*}$ through $Quad.$ 
It is not hard to show that a picture similar to one of the possibilities in Figure~\ref{oliveira1} 
must happen. More precisely, there always exists a Jordan curve separating 
the local $\widetilde{\lambda }_{1,s}$ from the hatched area of $Quad$.    

\begin{figure}[!ht]
\begin{center}
\includegraphics[width=10cm]{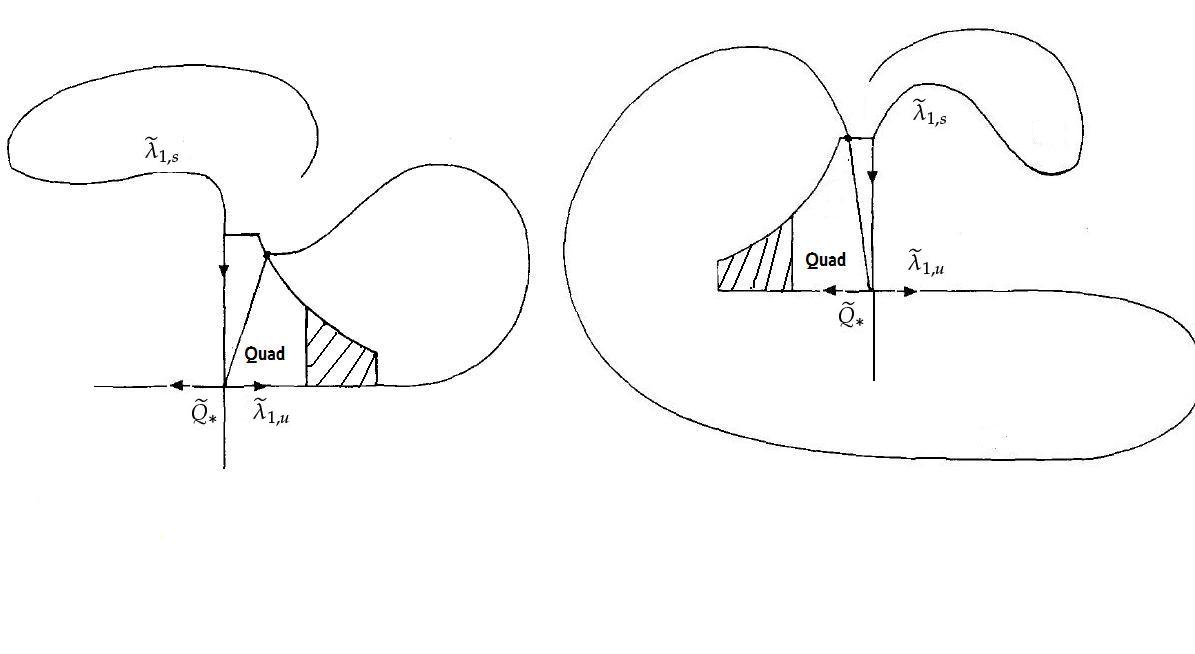}
\caption{The contractible case}
\label{oliveira1}
\end{center}
\end{figure}

And it is easy to see that the stable branch can not accumulate on 
$\widetilde{Q}^{*}$ through $Quad$ without intersecting the unstable branch: 
the only way it can enter $Quad$ is through the hatched area. And it has to intersect 
 the unstable branch in order to reach that area.
\end{proof}

\begin{Remark}\label{unbounded_for_saddle_like} Similar to the previous
  proposition, 
  the same conclusion holds for any index $0$ saddle-like fixed point.
  The proof follows 
the same lines of the above proof. 
\end{Remark}
%Now, let us denote the four branches of $\widetilde{Q}_{*}$ by $\widetilde{%
%\lambda }_{1,u},\widetilde{\lambda }_{2,u},\widetilde{\lambda }_{1,s}$ and $
%\widetilde{\lambda }_{2,s}$.

\begin{prop}\label{same_closure} 
The projections of the four branches are two by two disjoint and they
have the same closure, 
 denoted as follows. 
\begin{equation}
K  =  \text{cl}(\pi(\widetilde \lambda_{1,s})) = \text{cl}(\pi(\widetilde \lambda_{2,s}))
    =  \text{cl}(\pi(\widetilde \lambda_{1,u})) =  \text{cl}(\pi(\widetilde \lambda_{2,u})).
\end{equation}
Moreover, each connected component of the complement of $K$
is a periodic open disk. 
\end{prop}

\begin{proof}
Let us consider one branch, for instance, 
$\widetilde \lambda_{1,u}$, 
which is unbounded in $\Bbb R^2$ by 
Lemma~\ref{stable_unbounded}.
Then, the closure of its projection,
 $\text{cl}(\pi(\widetilde \lambda_{1,u}))$,
is an essential subset of $\Bbb T^2$. 
If  $\text{cl}(\pi(\widetilde \lambda_{1,u}))$ is not fully essential, 
then some connected component $U$ of its
complement is itself essential. 
Since $f$ is non-wandering, it follows that $U$ 
is periodic. But then, the rotation set $\rho(\widetilde f)$
has to be contained in some affine line with rational slope, which is 
a contradiction with assumption~(\ref{rotation_set_assumption}).
So, $\text{cl}(\pi(\widetilde \lambda_{1,u}))$ is fully essential. 
By item (1) of Definition~\ref{Gr_definition} and assumption (\ref{Gr_condition}),
the fixed point set of $f^n$ is finite, for all $n\geq 1$. So,  
Theorem~\ref{Bounded Invariant Disks} implies that 
every connected component of 
the lift of the complement of $\text{cl}(\pi(\widetilde {\lambda}_{1,u}))$
is a bounded $\widetilde{f}$-periodic disk. 
Thus, if $\widetilde \lambda$ is any branch (possibly the same branch),
and since $\widetilde \lambda$ is unbounded,
%it must intersect $\text{cl}(\widetilde \lambda_{1,u})$.
%So, 
then $\pi(\widetilde \lambda)$ 
intersects $\text{cl}( \pi (\widetilde \lambda_{1,u}))$.
If $\pi(\widetilde \lambda)$ also intersects $\text{cl}( \pi (\widetilde \lambda_{1,u}))^c$, then 
by
an argument very similar to that in the proof of Lemma~\ref{stable_unbounded} to treat possibility (1), 
we conclude that there exists a wandering domain inside an open periodic disk in the torus, a contradiction. Thus
$\pi(\widetilde \lambda)\subset \text{cl}(\pi (\widetilde \lambda_{1,u}))$. Since we have 
chosen the two branches arbitrarily, the proof is over. 
\end{proof}

%\begin{proof}[End of Proof of Theorem~\ref{generic_not_this_set}]
By Proposition~\ref{same_closure}, for the fixed saddle point $Q_{*}$ 
in the torus, each of its four branches accumulates on all the other
three branches, as well as on itself. Now we need to recall the final
arguments of the proof of Theorem 2 in ~\cite{Oliveira}. More
precisely, from page 591 to page 594 of~\cite{Oliveira}, the starting
conditions are that all branches are unbounded in $\Bbb R^2$, and the
closure of every branch in $\Bbb T^2$ accumulates on all the four branches. Under
these conditions, following exactly the arguments in that paper, we get that 
\begin{equation}
\pi (\widetilde{\lambda }_{1,u}\bigcup \widetilde{\lambda }_{2,u})\bigcap
\pi (\widetilde{\lambda }_{1,s}\bigcup \widetilde{\lambda }_{2,s})\neq
\emptyset 
\end{equation}
with a topologically transverse intersection.

So, either 
\begin{equation}
(\widetilde{\lambda }_{1,u}\bigcup \widetilde{\lambda }_{2,u})\bigcap ( 
\widetilde{\lambda }_{1,s}\bigcup \widetilde{\lambda }_{2,s})\neq \emptyset
, 
\end{equation}
or 
\begin{equation}
(\widetilde{\lambda }_{1,u}\bigcup \widetilde{\lambda }_{2,u})\bigcap ( 
\widetilde{\lambda }_{1,s}\bigcup \widetilde{\lambda }_{2,s}+(m,n)),\text{ for
some }(m,n)\in \Bbb Z^2\backslash (0,0), 
\end{equation}
in both cases, with topologically transverse intersections.

The first is a contradiction with Proposition~\ref{stable_unstable},
and for the second case, we can use a similar argument as in the proof of
Proposition~\ref{stable_unstable}, to create a non-contractible periodic
orbit. This is a contradiction with the assumption on the shape of $\rho ( 
\widetilde{f})$, i.e., (\ref{rotation_set_assumption}). %\end{proof}

To conclude, as we did in the proof of Lemma \ref{stable_unbounded}, we
present a sketch of the argument contained in the aforementioned pages of 
\cite{Oliveira}.

First, note that for any local quadrant, $Quad$ at $Q_{*},$ both adjacent
branches accumulate on $Q_{*}$ through it. Let $Quad$ be contained in 
the first quadrant and $\lambda _{1,u}$ and $\lambda _{1,s}$ be the branches 
adjacent to it.
Follow $\lambda _{1,u}$ starting at $Q_{*}$ until the first time it reaches $%
Quad,$ at a point $z\in \partial Quad.$ This could happen so that this sub
arc from $\lambda _{1,u}$ whose endpoints are $Q_{*}$ and $z,$ united with a
segment from $z$ to $Q_{*},$ is either a contractible loop, or not. If the
loop is contractible, then it bounds a disk $D$ that separates the hatched
area in the left case of Figure \ref{oliveira1} from a local part of $\lambda _{1,s}.$ As the
only way for $\lambda _{1,s}$ to enter $Quad$ is through the hatched area,
there must be an intersection between $\lambda _{1,s}$ and $\lambda _{1,u}.$

So we are left to consider the case when, for any of the four branches,
whenever it returns to some adjacent local quadrant $Quad_i$ $(i=1,2,3,4),$ it forms
a non-contractible loop ($Quad_1$ is in the first quadrant and so on, rotating 
counterclockwise). In this case, the situation in the universal cover is 
as in Figure \ref{oliveira2}.

\begin{figure}[!ht]
\begin{center}
\includegraphics[width=7cm]{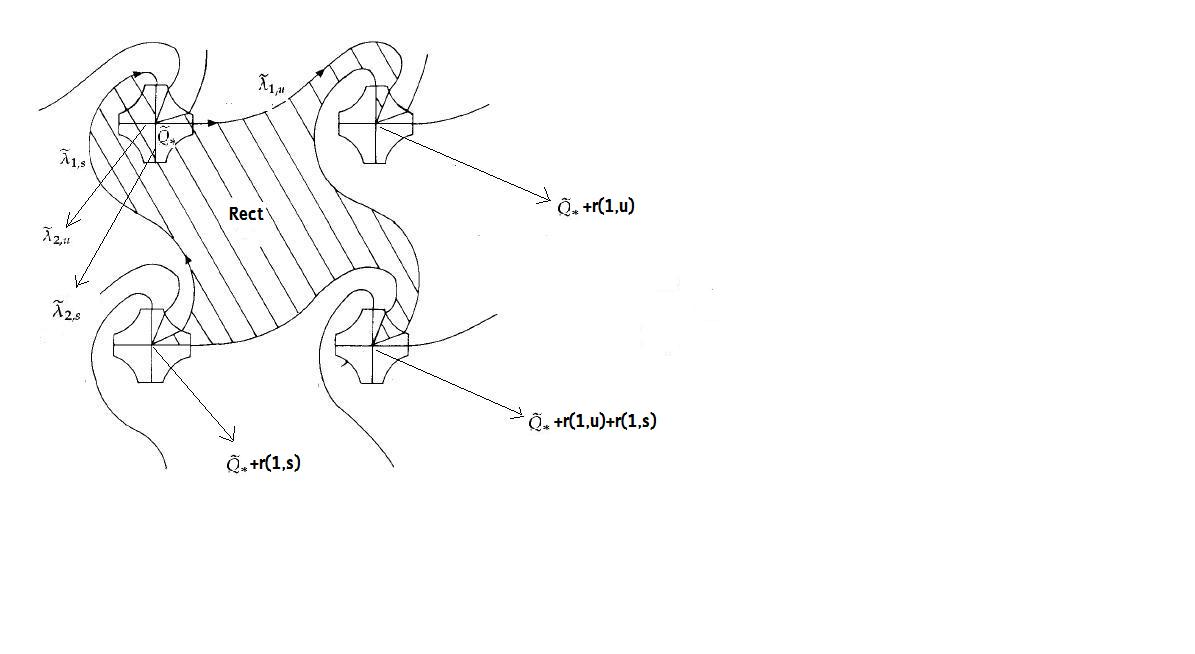}
\caption{The curvilinear rectangle Rect}
\label{oliveira2}
\end{center}
\end{figure}

There, we consider $\widetilde{\lambda }_{1,s}$ and $\widetilde{\lambda }%
_{1,u}$ starting at $\widetilde{Q}_{*}$ until the first point each of them
has in some connected component of $\pi ^{-1}($$\partial Quad_1)$ and then
back to some translate of $\widetilde{Q}_{*}$ through a segment: we get a
''web'' on the plane, whose building blocks are all the integer translates
of the curvilinear rectangle, denoted $Rect$ in Figure \ref{oliveira2}.

We know that $\widetilde{\lambda }_{2,s}$ and $\widetilde{\lambda }_{2,u}$
are both unbounded, so they have to leave $Rect.$ If they do not intersect
$\widetilde{\lambda }_{1,s}$ and $\widetilde{\lambda }_{1,u},$
the only possibilities
are, for $\widetilde{\lambda }_{2,u}$ it leaves $Rect$ through the connected
components of $\pi ^{-1}($$\partial Quad_1)$ that contain $\widetilde{Q}%
_{*}+r(1,u)$ or $\widetilde{Q}_{*}+r(1,u)+r(1,s).$ And $\widetilde{\lambda }%
_{2,s}$ leaves $Rect$ through the connected components of $\pi ^{-1}($$%
\partial Quad_1)$ that contain $\widetilde{Q}_{*}+r(1,s)$ or $\widetilde{Q}%
_{*}+r(1,u)+r(1,s).$ It is easy to see that, unless both $\widetilde{\lambda 
}_{2,s}$ and $\widetilde{\lambda }_{2,u}$ leave $Rect$ through the connected
component of $\pi ^{-1}($$\partial Quad_1)$ that contains $\widetilde{Q}%
_{*}+r(1,u)+r(1,s),$ the diagram in Figure \ref{oliveira2} implies that
there must be an intersection between $\widetilde{\lambda }_{2,s}$ and $
\widetilde{\lambda }_{2,u}$ and we are done.

So, assume this is the case. Now we fall in the situation described in
Figure \ref{oliveira3}.

\begin{figure}[!ht]
\begin{center}
\includegraphics[width=10cm]{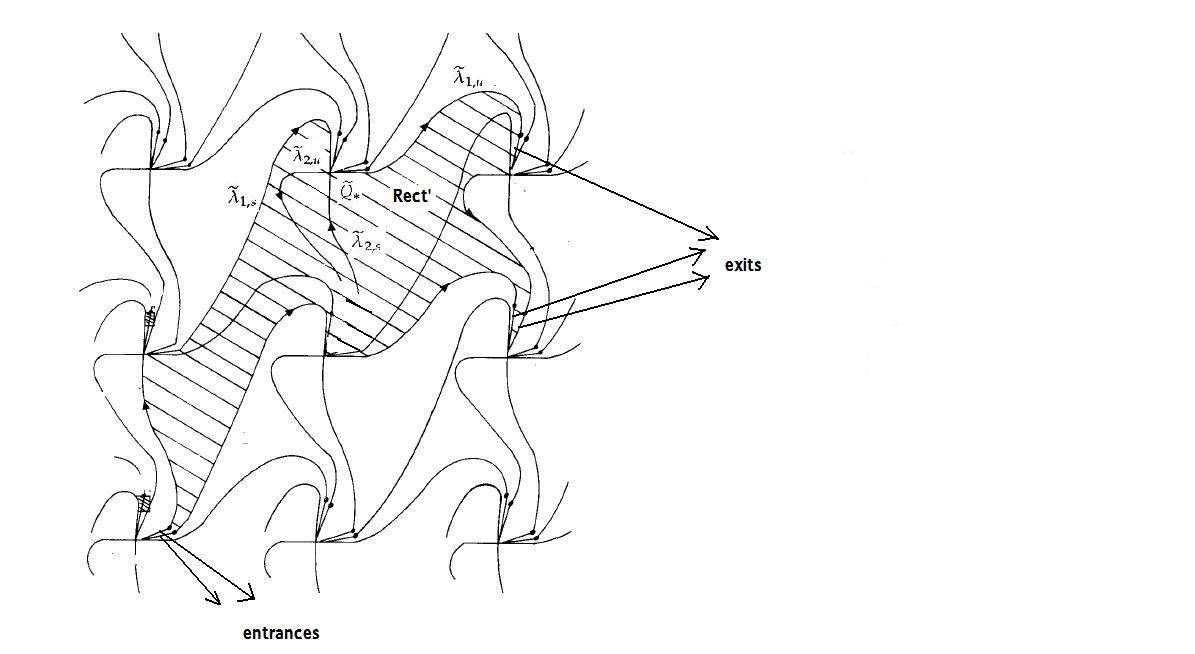}
\caption{The curvilinear rectangle Rect'}
\label{oliveira3}
\end{center}
\end{figure}

Again, as $\widetilde{\lambda }_{2,s}$ and $\widetilde{\lambda }_{2,u}$ are
both unbounded, they have to leave $Rect^{\prime }.$ If they do not intersect
$\widetilde{\lambda }_{1,s}$ and $\widetilde{\lambda }_{1,u},$
from the position
of the exits and entrances in $Rect^{\prime },$ there must be an
intersection between $\widetilde{\lambda }_{2,s}$ and $\widetilde{\lambda }%
_{2,u}$ (we are using the fact that unstable branches leave $Rect^{\prime }$
through one of the exits and stable branches leave $Rect^{\prime }$ through
one of the entrances). 

A final remark is that Figures 3,4 and 5 were taken from \cite{Oliveira}: we just
adapted them to our notation. \hfill $\Box $

\section{A Broader 
Class of  (non generic) Diffeomorphisms}
In the proof of Theorem~\ref{instability}, 
we keep the fixed points away from the support of 
the perturbation. 
Thus, the rotation set after the perturbation still contains the point 
$(0,0)$. 
On the other hand, it seems possible
that $(0,0)$ will be "mode locked" in the following sense. 
Possibly, for all sufficiently small perturbations, $(0,0)$
is not contained in the interior 
of the perturbed rotation set. 
This intuition comes from the phenomenon called 
rational mode locking. 
The case when the rotation set has non-empty interior
was treated in ~\cite{Addas_Calvez}.
One of the theorems proved there
states that rational mode locking happens 
under some conditions that are 
satisfied for generic one-parameter families.

This section has two objectives. 
First, we prove several results describing the dynamics of diffeomorphisms 
$f\in \mathcal K^r$, which together, imply Theorem \ref{carac_map}.
And then, using the previous results and some delicate topological arguments, we show the existence of
lots of Brouwer lines in the universal cover, which are lifts of essential loops
in the torus for all possible homotopy classes. 
In the end of the section we explain how Theorem~\ref{generic_family_not_perturbable} in the introduction 
can be deduced from the existence of Brouwer lines.

%\begin{proof}[Beguining of the Proof of Theorem~\ref{Thm_Generic_Family}]\renewcommand{\qedsymbol}{}
%Fix $f\in \mathcal K^r$, with a lift $\widetilde f$, 
%whose rotation set
%$\rho(\widetilde f)$  is the line segment from $(0,0)$ to $(\alpha,\beta)$.
%Choose the 
%integer vector $(a,b)$, where $a$ and $b$ are coprime. 
%Assume for definiteness that
%\begin{align}\label{alphabeta_abposition}
%\langle (\alpha,\beta),(a,b)\rangle  & >0.\\
%\langle (\alpha,\beta),(b,-a)\rangle & >0. \label{b-a_assumption}
%\end{align}
%Note that under this assumption, 
%if we consider the oriented line passing through $(0,0)$ and $(a,b)$, 
%then the vector $(\alpha,\beta)$ is to the right of this line. 
%The other cases can be treated similarly. 
%\end{proof}

We start by describing the dynamics in $\mathcal K^r.$ %First, we list several 
%properties about periodic points. 
In this whole section, $f \in \mathcal K^r$ and $\widetilde{f}$ is a lift 
of $f$ whose rotation set is the segment from $(0,0)$ to a totally irrational point
$(\alpha,\beta)$.  

\begin{prop}\label{corollary_about_periodic_non_fixed_point}
%Under the assumptions of the rotation set for $\widetilde f$, for $f\in \mathcal K^r$,
Every periodic point $p$ is indeed a fixed point,
with topological index $0$, and the local dynamics around $p$ 
can be described explicitly as in Lemma~\ref{section_two_lemma}. 
In particular, $p$ admits exactly one stable and one unstable branch. 
\end{prop} 

\begin{proof}  By the assumption on the shape of the rotation set $\rho(\widetilde f)$, 
all the periodic orbits of $f$
must be contractible, that is, they lift to periodic orbits of $\widetilde{f}$.
Suppose by contradiction that there exists a periodic point 
which is not fixed, or there exists a fixed point
which does not 
have topological index $0$. In the first case, by Lemma~\ref{Brouwer}, 
there exists some fixed point with positive 
topological index. 
On the other hand, it is easy to see from the definition of 
$\mathcal K^r$ that the indices at fixed points can only assume one of the 
following values: $-1,0$ or $1$.
Observing Lemma~\ref{Lefschetz}, as in both possibilities above 
there exists fixed points
with non-zero index, there must always exist a fixed point which has topological 
index $-1.$ And from the definition of $\mathcal K^r$ this point is a hyperbolic 
saddle.
%In
%the second case, 
%we can also obtain some fixed point which has topological index $-1$. 
By the same arguments used to prove 
Lemma~\ref{Saddle_Essential},
there exists a fixed hyperbolic saddle contained in $\overline{U},$
where $U$ is defined in expression \ref{defineu}. 
Now the proof goes exactly as in Theorem~\ref{generic_not_this_set},
 and as in that proof, we get a contraction with the shape of the 
rotation set $\rho(\widetilde f)$. 
So every fixed point has topological index $0$ and there
are no other periodic points.
Since $f\in \mathcal K^r$, the local dynamics around a fixed point is given by 
Lemma~\ref{section_two_lemma}.
\end{proof}

The next lemma is Corollary~\ref{coro_bounded}. 
Note that it depends on Theorem~\ref{case_of_0}.
\begin{lemma}[Corollary~\ref{coro_bounded} restated]\label{Generic_Bounded_in_(-alpha,-beta)} 
The lift $\widetilde f$ has bounded deviation along 
the direction $-(\alpha, \beta)$. Equivalently, there exists $M>0$, such that 
for any $\widetilde x$ and $n\geq 1$, 
\begin{equation}\label{Bounded_in_-v}
\text{pr}_{-(\alpha,\beta)} (\widetilde f^n(\widetilde x)-\widetilde x) \leq M.
\end{equation}
\end{lemma}

\begin{proof}
 By Proposition~\ref{AC}, for any $\widetilde g \in \widetilde{\text{Homeo}}_0(\Bbb T^2)$ 
which is a sufficiently small perturbation of $\widetilde f$,
 it is not possible that $\rho(\widetilde g)$ contains 
 $(0,0)$ in its interior.  
 Then  by Theorem~\ref{case_of_0}, $\widetilde f$ must have bounded deviation along $-(\alpha,\beta)$.
 \end{proof}

\begin{lemma}\label{density_unboundedness}
 For any $\widetilde{f}$-fixed point $\widetilde p$, 
 its stable and unstable branches are both unbounded. Their projections to  the torus  do not intersect. Moreover, the projection of each branch  is dense in the torus. 
  \end{lemma}
 
  \begin{proof}
    A first observation is that there is no periodic open disk. If such a disk existed, then from our hypotheses, its prime ends
    rotation number would be irrational. So $f$ would have periodic points with
    positive index (see the end of the proof of Lemma 5.2), something that is
    not allowed by Proposition 6.2.  

    Fix some fixed point $p$ in the torus. Consider any $\widetilde p$ in the plane that
    lifts $p$. %As there are no $f$-periodic open disks in the torus, the proofs
    The proofs of Proposition 5.3 and Lemma 5.4 imply that both
    the stable and the unstable branches at $\widetilde p$ are unbounded (see
    Remarks 5.1 and 5.2).
 
   Now we show that $W^s(p)$ is dense in $\Bbb T^2$
  (a similar argument works for $W^u(p)$).
  Since the lift $W^s(\widetilde p)$ is unbounded, the closure
  $\overline{W^s(p)}$
  must be an essential subset of $\Bbb T^2$.
  Assume its complement 
  $\Bbb T^2\backslash \overline{W^s(p)}$ 
  is non-empty. 
  If it contains an essential component, 
  then as we have already done in many previous arguments, 
  this implies that the rotation set is
  contained in a straight line with rational slope, a contradiction.
  Thus, $\Bbb T^2\backslash \overline{W^s(p)}$ is inessential. 
  So each connected component is a periodic open disk. As there are none, 
  $W^s(p)$ is dense in $\Bbb T^2.$

    Now, if the stable and unstable branches at $p$ intersect, then
    for some $\widetilde p$ lift of $p$, either its stable and unstable
    branches intersect,
    or the unstable branch at $\widetilde p$ intersects (maybe in a tangency) 
    the stable branch at
    $\widetilde p+(m,l)$ for some non-zero integer vector $(m,l)$.
    In this second case, there exists a Jordan curve in the plane, which is 
 given by the union of two arcs: one contained in the unstable branch at $\widetilde p$
and the other contained in the stable branch at $\widetilde p+(m,l)$.
As this Jordan curve bounds a disk and $W^s(p)$ is dense in $\Bbb T^2$, we get that
the stable branch at some integer translate $\widetilde p+(m',l')$ has a topologically transverse
intersection with the unstable branch at $\widetilde p$.
If $(m',l')$ is non-zero, Lemma 2.10 and Remark 2.5 give a contradiction.

So we are left to consider the case when the stable and unstable branches 
at $\widetilde p$ intersect. As we did before, there exists an open disk
 $\widetilde D$ in the plane
 whose boundary is a Jordan curve containing $\widetilde p$, 
  consisting of two compact arcs. 
 One of the arcs is contained in the unstable branch at 
 $\widetilde p$  and the other is contained in the stable. Now, either
 for some integer $n>0,$ $\widetilde{f}^n(\widetilde D)$
  intersects some non-zero integer translate of $\widetilde D,$ something that
  is not allowed, again by Lemma 2.10 and Remark 2.5, or not. In this second
  possibility,  if we consider
  $\widetilde D_{\text{sat}} =\bigcup_{n\in \Bbb Z}
  \widetilde f^n(\widetilde D),$ then
 $Fill(\widetilde D_{\text{sat}})$ is open and disjoint from all its non-zero
  translates. Therefore, when projected to the torus, it is an $f$-invariant
  bounded open disk (see Lemma 2.3). As such disks do not exist, there are
  no homoclinic intersections in the torus.
  %As there are no connections between
  % branches of saddle-like periodic points, by item (i) of
  % Lemma~\ref{primeends},
  %the prime ends rotation number of any of these open disks must be irrational. %
  %And by an argument similar to the one which appeared in 
  %the proof of Lemma~\ref{Saddle_Essential},
  %there exists 
  %a fixed point whose topological index is positive, a contradiction with
  %Proposition~\ref{corollary_about_periodic_non_fixed_point}.
%The proof of the first two claims already appeared 
% in the Remarks~\ref{no_homoclinic_intersection_saddle_like}
% and~\ref{unbounded_for_saddle_like}. 
  %Since $f$ has no saddle-like 
  %connection, by item (i) of Lemma~\ref{primeends},
  %the prime ends rotation number of any of these disks must be irrational. 
  %However, by a similar argument which appeared in 
  %the proof of Lemma~\ref{Saddle_Essential},
  %this implies the existence of 
  %a fixed point whose topological index is $1$, which is a contradiction with Proposition~\ref{corollary_about_periodic_non_fixed_point}.
  \end{proof}

 The next lemma shows that $f$ also does not admit heteroclinic intersections. 

 \begin{lemma}\label{no_heteroclinic_points}  
 For any two 
fixed points 
 $p_1$ and $p_2$, 
 we have
 \begin{equation}
 W^u(p_1)\bigcap W^s(p_2)=\emptyset.
 \end{equation}
 \end{lemma}
 \begin{proof} 
Suppose $f$ admits some fixed points $p_1$ and $p_2$, 
and 
\begin{equation}
W^u(p_1) \bigcap W^s(p_2) \neq \emptyset. 
\end{equation} 
We consider some lifts $W^u(\widetilde{p_1})$ and $W^s(\widetilde{p_2})$
of these branches, 
such that their intersection is non-empty. 
Then, since there are no saddle-like connections, 
we can find some Jordan curve, 
which is the union of one
sub-arc of 
 $W^u(\widetilde{p_1})$ 
 and one sub-arc of $W^s(\widetilde{p_2})$, respectively. 
 This  
 Jordan curve bounds a topological disk, 
 denoted  $\widetilde U$. 
 The projection
 $U=\pi(\widetilde U)$ is a proper 
 open subset of $\Bbb T^2$.

Since both $W^u(p_1)$ and $W^s(p_2)$ are dense in $\Bbb T^2$, 
each of them must
 intersect  $U$. 
 So, there are homoclinic intersections, which do not exist by
  the previous Lemma. This contradiction ends the proof.
  \end{proof}
 
 The goal now is to show that both invariant branches at a fixed point tend to infinity. 
 
 \begin{lemma}\label{unstable_goes_to_infinity} 
 Let $\widetilde p$ be a fixed point (for $\widetilde{f}$). 
 Then both its stable and unstable branches 
 intersect every compact set 
 in a closed subset. 
 More precisely, 
 for the unstable branch $W^u(\widetilde p)\backslash\{\widetilde p\}$
  for example, 
 if $\widetilde \lambda \subset W^u(\widetilde p)\backslash\{\widetilde p\}$ 
denotes the closure of a fundamental domain, 
 then for any compact set $K$, 
 the set $\{n\geq 1 \big | \widetilde f^n(\widetilde \lambda)\bigcap K \neq \emptyset\}$ is finite.  
 % there exists $N=N(K)$ such that, 
 % for any $n\geq N$,
 % $\widetilde f^n(\widetilde \lambda)\bigcap K=\emptyset$. 
 \end{lemma}
 
 \begin{proof}
 Consider the unstable branch $W^u(\widetilde p)\backslash\{\widetilde p\}$ 
  and choose a fundamental domain contained in it, 
whose closure we denote by $\widetilde \lambda$. 
  Suppose by contradiction that there exists some compact set 
$K\subset \Bbb R^2$, 
an integer sequence $n_i \to +\infty$, 
and a sequence $\widetilde q_i \in \widetilde \lambda$, 
such that 
\begin{equation}
\widetilde f^{n_i}(\widetilde q_i)\in K.
\end{equation}
 By extracting a subsequence if necessary, 
 we can assume 
 $\widetilde q \in \lambda$ is the limit point of the sequence 
 $\{\widetilde q_i\}_{i\geq 1}$.
 Considering the $\omega$-limit set of $\widetilde q$, denoted
 $\omega(\widetilde q)$, there are three possibilities:

Either $\omega(\widetilde q))$  is  empty, a singleton, or it has more than
 one point. 
 If it contains more than one point, then some point  $\widetilde w$ in it
 is not fixed, because each fixed point is isolated. 
 Then $\widetilde w$ is contained in some disk $U$, 
 such that, 
\begin{align}
   \widetilde f(U)\bigcap U & = \emptyset, \\
\widetilde f^k(U)\bigcap U & \neq \emptyset, \text{ for some } k\geq 2.
\end{align} 
So Lemma~\ref{Brouwer} implies that $\widetilde f$ admits some fixed point
with positive topological index, a 
 contradiction with Proposition~\ref{corollary_about_periodic_non_fixed_point}.

If $\omega (\widetilde q)$ is a singleton, say, 
 equal to $\{\widetilde r\}$, then $\widetilde r$ is necessarily a fixed point. 
 So $\widetilde q$ belongs to the stable branch  
 at the point $\widetilde r$, that is, there is an heteroclinic point,
 a contradiction with Lemma 6.5. 

So, $\omega(\widetilde q)$  is empty.

%
 %
 %If $\omega(\widetilde q)\neq \emptyset$ contains more than one point, 
 % then some point
 %$\widetilde z\in \omega(\widetilde q)$ is not fixed, because each fixed point is isolated. 
 %Then $\widetilde z$ is contained in some disk $U$, 
 %such that, 
 %\begin{align}
 %    \widetilde f(U)\bigcap U & = \emptyset, \\
 %\widetilde f^k(U)\bigcap U & \neq \emptyset, \text{ for some } k\geq 2.
 %\end{align} 
 %Then by Lemma~\ref{Brouwer}, 
 %it follows that $\widetilde f$ admits some fixed point with positive
 %topological index, a 
 %contradiction with Proposition~\ref{corollary_about_periodic_non_fixed_point}.
%
% If $\omega (\widetilde q)$ is a singleton, say, 
% equal to $\{\widetilde r\}$, then $\widetilde r$ is necessarily a fixed point, 
% so $\widetilde q$ belongs to the stable branch  
% of the point $\widetilde r$. By Lemma~\ref{no_heteroclinic_points}, this is not possible. 
% 
 %As in the proof of that lemma, we conclude that $\omega(\widetilde q)=\emptyset$. 
 In particular, 
 this means that  the sequence $\{\widetilde f^n(\widetilde q)\}_{n\geq 1}$ converges 
 to infinity as 
 $n$ tends to $+\infty$.  
 By Theorem~\ref{Fabio_Bounded} and Lemma~\ref{Generic_Bounded_in_(-alpha,-beta)}, 
 $\widetilde f$ has bounded deviation along the three
 directions $-(\alpha, \beta), (-\beta,\alpha)$ and $(\beta,-\alpha)$. 
So the sequence 
$\widetilde f^n(\widetilde q)$ converges to infinity 
 along the direction $(\alpha,\beta)$.  

In particular, 
for some large $n_0$,
\begin{equation}
\inf_{\widetilde z\in K}  \big( \text{pr}_{(\alpha,\beta)} 
( \widetilde f^{n_0}(\widetilde q)-  \widetilde z) \big)>2M,
\end{equation} 
where $M>0$ comes from estimate (\ref{Bounded_in_-v}).

Then there exists some small disk 
$B$ containing $\widetilde q$, such that, 
for every point $b\in B$, the above estimate also holds true.  
Now, we can choose a sufficiently large $i$, with 
$\widetilde q_i\in B$, and $n_i > n_0$. 
Thus, 
\begin{align}
         & \text{pr}_{(\alpha,\beta)} \big( \widetilde f^{n_0} (\widetilde q_i) - \widetilde f^{n_i}(\widetilde q_i)\big) \\
\geq  & \inf_{\widetilde z\in K}  \big( \text{pr}_{(\alpha,\beta)} 
( \widetilde f^{n_0}(\widetilde q_i)-  \widetilde z) \big) \\
\geq & 2M.
\end{align}
As this is a contradiction with Lemma \ref{Generic_Bounded_in_(-alpha,-beta)}, 
the proof is completed. 
 \end{proof}

\begin{proof}[Proof of Theorem~\ref{carac_map}]\renewcommand{\qedsymbol}{}

The proof now follows easily from 
Proposition \ref{corollary_about_periodic_non_fixed_point}, 
Lemmas \ref{density_unboundedness}, \ref{no_heteroclinic_points}, 
\ref{unstable_goes_to_infinity} and Theorem \ref{Fabio_Bounded}.

\end{proof}

\begin{theorem}\label{Thm_Generic_Family} 
Let $\widetilde f$ denote some lift of some
$f\in \mathcal K^r$, and suppose 
$\rho(\widetilde f)$ is the line segment from 
$(0,0)$ to $(\alpha,\beta)$. 
Then, for any coprime 
integer pair $(a,b)\neq (0,0)$, 
there exists a torus loop $\ell= \ell_{(a,b)}$,
 which can be lifted to 
an $\widetilde f$-Brouwer line $\widetilde \ell$, such that 
$\widetilde \ell+(a,b)=\widetilde \ell$.
\end{theorem}

\begin{proof}

Up to a change of coordinates and/or considering $f^{-1}$ if necessary, we reduce to the 
case when $(a,b)=(0,1)$ and $\alpha >0.$

 \begin{prop}\label{good_curve} 
   There exists an oriented properly
   embedded curve $\widetilde \gamma \subset \Bbb R^2$, 
 with the following properties.
 \begin{enumerate}
 \item  $\widetilde \gamma+(0,1)=\widetilde \gamma$,
   and  $\widetilde \gamma$ is oriented in the direction $(0,1)$.
 \item $\widetilde \gamma$ does not contain any $\widetilde f$-fixed points. 
 \item Let
 $\mathcal R(\widetilde \gamma)$ denote the unbounded complementary domain to the right of 
 $\widetilde \gamma$. For any $\widetilde f$-fixed point 
 contained in $\mathcal R(\widetilde \gamma)$,
 its unstable branch does not intersect $\widetilde \gamma$.
 \item Analogously, let
 $\mathcal L(\widetilde \gamma)$ 
 denote the unbounded complementary domain to the left 
 of 
 $\widetilde \gamma$. 
 For any $\widetilde f$-fixed point 
 contained in $\mathcal L(\widetilde \gamma)$,
 its stable branch does not intersect $\widetilde \gamma$.
 \end{enumerate}
 \end{prop}
 \begin{proof}   
% Up to change of coordinates, we reduce to the case
 %$(a,b)=(0,1)$, and $(\alpha,\beta)$ satisfies 
 %$\alpha>0$. 
 
 Start with a vertical line $\ell$, oriented upwards, 
which does not contain any $\widetilde f$-fixed point. 
The complement, $\ell^c,$
consists of two unbounded connected components. 
We denote by $\mathcal R(\ell)$ (respectively, $\mathcal L(\ell)$)
the right component (respectively, the left component).
Let $O_-$ (respectively, $O_+$) 
denote 
the union of the 
stable branches (respectively, unstable branches) of all the $\widetilde f$-fixed points belonging 
to $\mathcal L(\ell)$ (respectively, $\mathcal R(\ell)$). 
We claim that both $O_-$ and $O_+$
 are closed sets.
  The arguments are similar, so it suffices to 
  prove the claim for 
  $O_-$.  

Let $\{z_i\}_{i\geq 1}$ 
be a sequence of points in $O_-$ 
converging to some point $z$. 
Choose a small closed disk 
$B(z,1)$
containing $z$.
By Theorem~\ref{Fabio_Bounded}, Lemmas 6.2 and ~\ref{unstable_goes_to_infinity}, 
there are only finitely many $\widetilde f$-fixed points in $\mathcal L(\ell)$, whose 
stable branches  
intersects $B(z,1)$. 
Moreover, 
the intersection of one such stable branch with $B(z,1)$
is a closed set.  So $O_-\bigcap B(z,1)$ is closed. 
Therefore, $z$ belongs to $O_-$, 
which implies that $O_-$ is closed.

It is clear that $(O_-)^c$
has a connected component 
which is unbounded to the right. More precisely, 
this component contains some translated domain
$\mathcal R(\ell)+(M',0)$, where $M'$ is a positive constant
obtained from the constant $M$ in
(\ref{Bounded_in_-v}) by defining $M'=M\cos \theta$, where $\theta$
is the angle between the horizontal line and the vector $(\alpha,\beta)$. 

Suppose 
by contradiction that $(O_-)^c$ is not connected. 
Then there exists a connected component $B$, 
which is contained in $\mathcal L(\ell)+(M',0)$. 
Observe $B$ is open and its
boundary $\partial B$ is contained in $O_-$, 
and recall that the unstable branch of every 
$f$-fixed point is dense in $\Bbb T^2$. 
It follows that, 
some unstable branch of some $\widetilde f$-fixed point 
intersects $B$. 
Since the branch is unbounded to the right, 
it must
intersect the boundary of $B$, 
which is a contradiction, because stable and
unstable branches do not intersect.
So, $(O_-)^c$ is connected. The same holds true for $(O_+)^c$.

Now consider the one point compactification of the plane, identified with $\Bbb S^2$.
The two closed sets $O_-$ and $O_+$ lift to $\widehat{O_-}$ and $\widehat {O_+}$, 
respectively, which are connected closed subsets, 
 because every stable and unstable branch lifts to some closed set containing 
the point $\infty\in S^2$. Clearly, $\widehat{O_-} \bigcap \widehat {O_+} =\{\infty\}$. 
By Lemma~\ref{Newman}, 
the complement of $O_- \bigcup O_+$ is
an open connected subset of $\Bbb R^2$.
Note that, if a point $\widetilde z\in (O_- \bigcup O_+)^c$,
then $\widetilde z+(0,k)\in (O_- \bigcup O_+)^c$ 
for any $k\in \Bbb Z$, 
because of
the relations 
$O_-=O_-+(0,1)$ 
and $O_+=(O_+)+(0,1)$.

So, we can choose an arc $\delta$ 
connecting $\widetilde z$ and $\widetilde z+(0,1)$ 
such that 
\begin{align}
\delta \bigcap (O_- \bigcup O_+) & =\emptyset \text{ and } \\
\delta \bigcap (\delta+(0,1)) & \text{ contains exactly one point}.
\end{align}
Therefore, the union 
\begin{equation}
\widetilde \gamma:= \bigcup_{i\in \Bbb Z}( \delta+(0,i)).
\end{equation} 
is a properly embedded real line, which satisfies all four properties. 
 \end{proof}
 
 \begin{lemma}\label{good_neighbourhood} Let $\widetilde \gamma$ be obtained from Proposition~\ref{good_curve}.
 For any fixed point $\widetilde q\in\mathcal R(\widetilde \gamma)$, 
 there exists a small closed neighbourhood 
 $\widetilde K$ containing $\widetilde q$, such that
 $\{\widetilde f^n(\widetilde K)\}_{n\geq 0} \subset \mathcal R(\widetilde \gamma)$.
 In particular, the forward iterates of $\widetilde K$ do not intersect 
 $\widetilde \gamma$.
 \end{lemma}
 \begin{proof}

Fix some fundamental domain $\widetilde{\lambda }$ of $W^u(\widetilde{q})$
very close to $\widetilde{q},$ whose endpoints are $\widetilde{y}$ and $
\widetilde{f}(\widetilde{y}).$ As $W^u(\widetilde{q})$ is unbounded to the
right, there exists $N>0$ such that 
$\text{dist}(\widetilde{f}^N(\widetilde{\lambda 
}),\widetilde{\gamma })>2M+C_{\widetilde f},$ where $M$ is the constant obtained in 
(\ref{Bounded_in_-v}), and $C_{\widetilde{f}}$ is given by:

\begin{equation}\label{C_tildef_definition}
C_{\widetilde f}=\max_{\widetilde z \in \Bbb R^2} \|\widetilde f(\widetilde z) -\widetilde z\|.
\end{equation}

\begin{figure}[!ht]
\begin{center}
\includegraphics[width=7cm]{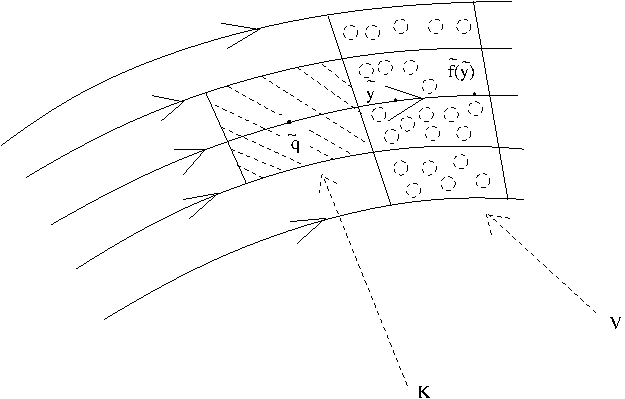}
\caption{The Neighbourhoods $K$ and $V$}
\label{figlemma68}
\end{center}
\end{figure}

Note that $\bigcup_{n=0}^\infty \widetilde f^n(\widetilde \lambda)$ does not intersect $\widetilde \gamma$. 
We can choose a small open neighbourhood $V$ of $\widetilde{\lambda }$,
such that 
$\bigcup_{n=0}^N$$\widetilde{f}^n(V)$ is sufficiently close to 
$\bigcup_{n=0}^N 
\widetilde{f}^n(\widetilde{\lambda })$, so that 
it does not intersect $
\widetilde{\gamma }$. 
Observing Lemma~\ref{Generic_Bounded_in_(-alpha,-beta)}, 
we can also ensure that
$\widetilde{f}^n(V)$
does not intersect $\widetilde{\gamma },$ for all $n>N$. 
Finally, choose a small neighbourhood $K$ of $\widetilde q$, 
such that, for every point in $K$, either it belongs to the stable branch of $\widetilde q$, 
or it has a forward iterate belonging to $V$
 (See Figure~\ref{figlemma68} for the choices of these neighbourhoods).
 In fact, 
 $\bigcup_{n=0}^\infty \widetilde f^n(K)\subset K\bigcup (\bigcup_{n=0}^\infty \widetilde f^n(V))$. 
 Therefore, all non-negative iterates of $K$ can not leave 
 $\mathcal R(\widetilde{\gamma }).$     \end{proof}

  \begin{lemma}\label{N_times_becomes_free}
  Let $\widetilde \gamma$ be the curve obtained in 
  Proposition~\ref{good_curve}.
  Then there exists
  a positive integer $N$ such that, 
  \begin{equation}\label{gamma_is_free}
  \widetilde f^N(\widetilde \gamma)\bigcap \widetilde \gamma =\emptyset.
  \end{equation}
  \end{lemma}
  
  \begin{proof}
  Suppose by contradiction that
  there exists some sequence
  of points $\widetilde z_{(n)} \in \widetilde \gamma$ 
  such that $\widetilde f^n(\widetilde z_{(n)})\in \widetilde \gamma$. 
  Noticing item (1) of Proposition~\ref{good_curve}, 
  we can choose all the $\widetilde z_{(n's)}$ in a compact fundamental domain of
  $\widetilde{\gamma}$, denoted $K.$ In particular, 
  they have an 
  accumulation point 
  $\widetilde z_\ast$. 
  Up to extracting 
  a subsequence, simply 
  assume $\widetilde z_{(n)} \to \widetilde z_\ast$.
  
Moreover, from Theorem~\ref{Fabio_Bounded}, there exists a constant $M^*>0$ such 
that, for all integers $n>0,$ $\widetilde{f}^n(K) \cap \widetilde{\gamma}$ is contained in the 
$M^*$-neighbourhood of $K$ in $\widetilde{\gamma}$.

  By Theorem~\ref{Fabio_Bounded} and
  Lemma~\ref{Generic_Bounded_in_(-alpha,-beta)},
  the forward orbit $\{\widetilde f^n(\widetilde z_\ast)\}_{n\geq 0}$
  is bounded in three 
  directions $(-\alpha,-\beta), (-\beta,\alpha), (\beta,-\alpha)$.
  Then, either $\{\widetilde f^n(\widetilde z_\ast)\}_{n\geq 0}$
  is unbounded in the direction 
  $(\alpha,\beta)$, or it is bounded. 
  We seek contradictions in both cases. 
  
  If $\{\widetilde f^n(\widetilde z_\ast)\}_{n\geq 0}$ is bounded, 
  then the omega limit set $\omega(\widetilde z_\ast)$
  must be a single fixed point, otherwise, by the arguments 
  used in the proof of Lemma~\ref{unstable_goes_to_infinity}, 
  one finds a positive index fixed point, which does not exist by 
  Proposition~\ref{corollary_about_periodic_non_fixed_point}.
  So $\widetilde z_\ast$ 
  belongs to some stable branch $W^s(\widetilde q)$,
  for some fixed point 
$\widetilde q \in \mathcal R(\widetilde \gamma)$. 
  By Lemma~\ref{good_neighbourhood},
  there exists a compact neighbourhood $\widetilde K$ of $\widetilde q$, 
  whose forward iterates do not intersect $\widetilde \gamma$. And
  there exists a positive integer $m_0$, such that 
  $\bigcup_{k=0}^{m_0} \widetilde f^{-k}(\widetilde K)$ 
  contains some neighbourhood $N$
  of 
  $\widetilde z_\ast$.  This is a contradiction, because for sufficiently 
  large  integers $m > m_0$, 
  $\widetilde z_{(m)}\in N$
  and $\widetilde f^m(\widetilde z_{(m)})\in \widetilde \gamma$.
 
  The other case is when the orbit of $\widetilde z_\ast$ is unbounded. 
  Then, for sufficiently large $k_0$, 
  \begin{equation} 
  \inf_{\widetilde w\in K} \text{pr}_{(\alpha,\beta)} 
  (\widetilde f^{k_0}(\widetilde z_\ast)- \widetilde w) > 10(M+M^*+1).
  \end{equation} 
  Then, when $m\geq k_0$ is sufficiently large, 
      \begin{equation}
    \text{pr}_{(\alpha,\beta)} 
    \big(\widetilde f^{k_0}(\widetilde z_{(m)})-\widetilde z_{(m)} \big) >10(M+M^*+1).
    \end{equation}
    Noticing  Lemma~\ref{Generic_Bounded_in_(-alpha,-beta)}, 
    this provides a contradiction, 
    since 
$$
\widetilde f^m(\widetilde z_{(m)})\in M^* \text{-neighbourhood of } K.
$$
    The proof is completed now. 
  \end{proof}

    For the oriented curve $\widetilde \gamma$ from 
    Proposition~\ref{good_curve} above,
    $\mathcal R(\widetilde \gamma)$ denotes 
    the unbounded connected component of $(\widetilde  \gamma)^c$ 
  in the direction of $(\alpha,\beta)$. In 
  Lemma~\ref{N_times_becomes_free}, we have obtained the integer $N$ such that 
  $\widetilde f^N(\widetilde \gamma) \subset \mathcal R(\widetilde \gamma)$.
    %of the curve $\widetilde \gamma$. 
  
  The following is a standard argument. Consider the finite union of 
  curves, 
  \begin{equation}
  Q:= \bigcup_{j=0}^{N-1} \widetilde f^j(\widetilde \gamma).
  \end{equation}
  Clearly, the complement of $Q$ has a component which is unbounded in the
  direction of $(\alpha,\beta)$. If   
  $\widetilde \ell$ is the boundary of this component, then 
  $\widetilde f(\widetilde \ell) \cap \widetilde \ell=\emptyset.$
This $\widetilde \ell$ is clearly the lift of a vertical loop in the torus. And so the proof of Theorem \ref{Thm_Generic_Family} is over.
   \end{proof}

We close this section by restating 
and proving  the remaining part of Theorem~\ref{generic_family_not_perturbable}.

\begin{theorem}[Remains of Theorem~\ref{generic_family_not_perturbable}] 
Let $\widetilde f$ denote some lift of 
$f\in \mathcal K^r$, and 
$\rho(\widetilde f)$ is the line segment from 
$(0,0)$ to $(\alpha,\beta)$.
Let $\gamma$ be any 
straight line passing through 
$(0,0)$, which does not contain $\rho(\widetilde f)$.
Then there exists $\varepsilon_0>0$ 
such that for any 
$\widetilde g\in \widetilde{ \text{Homeo}}_0(\Bbb T^2)$, 
 which is $C^0$-$\varepsilon_0$-close to $\widetilde f$, 
the rotation set $\rho(\widetilde g)$ does not intersect the 
connected component of $\gamma^c$ which does not intersect $\rho(\widetilde f)$.
\end{theorem}

\begin{proof}
Choose two reduced integer vectors, $(a,b)$ and $(a',b')$, with the following properties. 
\begin{enumerate}
\item the two rays from $(0,0)$ in the directions $(a,b)$ and $(a',b')$ define 
a closed cone $C$ 
which contains the vector $(\alpha,\beta)$ in its interior. 
\item the interior of $C$ is contained in one of the connected component of $\gamma^c$.
\end{enumerate}
By Theorem~\ref{Thm_Generic_Family}, there are two $\widetilde f$-Brouwer lines $\widetilde \ell_1$ and $\widetilde \ell_2$, 
such that, $\widetilde \ell_1+(a,b)=\widetilde \ell_1$, and $\widetilde \ell_2+(a',b')=\widetilde \ell_2$. Since both
$\widetilde \ell_1$ and $\widetilde \ell_2$ are 
lifts of simple closed curves in $\Bbb T^2$,  
there exists $\varepsilon_0$, 
such that, for any $\widetilde g\in \widetilde{\text{Homeo}}_0(\Bbb T^2)$,
with $\text{dist}_{C^0}(\widetilde g,\widetilde f)<\varepsilon_0$, 
those two lines $\widetilde \ell_1$ and $\widetilde \ell_2$ are still Brouwer lines for $\widetilde g$. 
This implies that $\rho(\widetilde g)\subset C$. 
In particular, $\rho(\widetilde g)\backslash \{(0,0)\}$
is contained in the connected component of $\gamma^c$ which contains 
$\rho(\widetilde f)\backslash \{(0,0)\}$.
\end{proof}

\section{Unbounded Deviations}
In this section, we show the following theorem.

\begin{theorem}\label{unbounded_final}[Theorem~\ref{generic_bounded_unbounded} restated]
Suppose $\widetilde f$ is a lift of some $f\in \mathcal K^r$, and 
$\rho(\widetilde f)$ is a segment from $(0,0)$
to a totally irrational point $(\alpha,\beta)$. 
Assume further that 
$f$ preserves a 
%$C^1$ 
foliation on $\Bbb T^2$.
Then $\widetilde f$ has 
unbounded deviation along the direction $(\alpha,\beta)$.
\end{theorem}

\begin{proof}
Let us assume by contradiction that
 there exists $M>0$, such that
 \begin{equation}\label{Bounded_by_S}
\sup_{\widetilde x \in \Bbb R^2, n\geq 1} 
\text{pr}_{(\alpha,\beta)} 
\big ( \widetilde f^n (\widetilde x) -\widetilde x -n(\alpha,\beta) \big) 
\leq M. 
 \end{equation}
Recalling Definition~\ref{Mab_Sab},
$\mathcal S_{(\alpha,\beta)}$ is the closure of the union of the
support of all the $f$-invariant ergodic probability 
measures whose average rotation vector for $\widetilde f$
is $(\alpha,\beta)$. 
Then, 
Lemma~\ref{geq_-M} shows that
 for any lift $\widetilde x\in \Bbb R^2$ of 
 some $x\in \mathcal S_{(\alpha,\beta)}$,  
and any $n\geq 1$,
  \begin{equation}\label{Bounded_Below_again}
   \text{pr}_{(\alpha,\beta)} \big(\widetilde f^n(\widetilde x) -\widetilde x -n(\alpha,\beta) \big) \geq -M.
      \end{equation}
      Clearly, for any fixed point $p$, there exists a small disk 
      $B$ containing 
      $p$, such that for any point in $B$, its $f$-iterates will 
       remain close to $p$ for a long time, both 
       in the future and in the past. 
       Expression (\ref{Bounded_Below_again}) 
      implies immediately that, 
      when $B$ is sufficiently small,
      the whole orbit of an arbitrary point in $\mathcal S_{(\alpha,\beta)}$ does not intersect 
       $B$.

Choose a fixed point $p$.
Since $f$ preserves the foliation $\mathcal F$, 
then the leaf 
$\mathcal F(p)$
containing $p$ must be the union of $W^s(p)$ and $W^u(p)$. 
Choose a local leaf $L\subset \mathcal F(p)$, 
which connects some point $y\in W^s(p)$ and 
$y'\in W^u(p)$.  Let $V$ be an open neighbourhood of $L$ such that for
any local leaf in $V,$ its forward and backward iterates under $f$ also intersect $V$.
Choose two small arcs $\gamma$ and $\gamma'$, both contained in $V$, transverse to the local foliation restricted to $V$, such that  
the arc $\gamma$ connects $y$ 
to a point $x$ and $\gamma'$ connects $y'$ to a point $x'$. Moreover, $f(\gamma)$ and $f^{-1}(\gamma')$ are also both contained in $V$
and $x$ and $x'$ bound a local leaf $\theta^+$. 
%and $\theta^+$ lies above the leaf $L$.
It is also convenient to choose $\theta^+$ so that it belongs either to 
$W^u(p)$ or $W^s(p)$. This is possible because both branches are dense in $\Bbb T^2$, see Lemma \ref{density_unboundedness}.
% which lies above the arc $L$.
Denote the 
closed region bounded by 
$\gamma,\theta^+,\gamma',L$ as $K$
(See Figure~\ref{fig7a}). Note that $K$ can be chosen 
arbitrarily close to $L$.
 
\begin{figure}[!ht]
\begin{center}
\includegraphics[width=8cm]{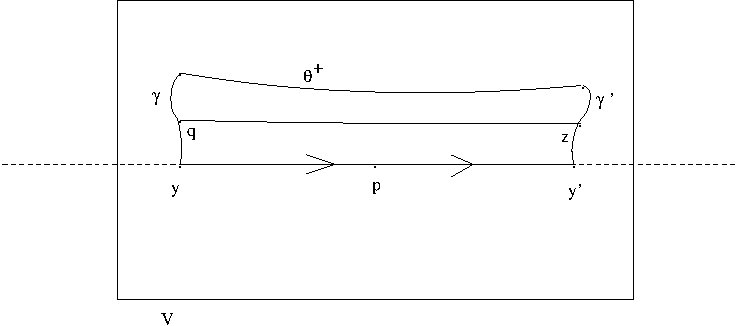}
\caption{Local Foliation around $p$.}
\label{fig7a}
\end{center}
\end{figure}
 
The claim is that, 
some local leaf in $K$, 
which is contained in $W^s(p)$ or $W^u(p)$, 
and different from $L$,
must contain a fundamental domain of $W^s(p)$ or $W^u(p)$,
that is, 
some sub-arc connecting a point and its image. 
Suppose by contradiction that this is not true. 

Then one of the following cases must happen (see Figure~\ref{fig7a}). 
\begin{enumerate}
\item $f(\theta^+)$ intersects $\theta^+$.
\item $f(\theta^+)$ is above $\theta^+$.
\item $f(\theta^+)$ is below $\theta^+$.
\end{enumerate}

If case (1) happens, since  $\theta^+$ belongs to  
$W^s(p)$ or $W^u(p)$, 
then it contains a fundamental domain of $W^s(p)$ or $W^u(p)$ and 
we are done. 

Up to considering the backward dynamics and 
exchanging the roles of 
stable and unstable branch, 
we can simply assume $f(\theta^+)$ is below $\theta^+$.
Then, $f(\gamma)$ is contained in $K$, provided 
the region is chosen sufficiently close to $L$. 

Since $W^s(p)$ is dense, we can 
follow it until the first time it enters the region $K$.
Denote  by $z \in W^s(p)$
the first entering point $(z \in \gamma')$. 
The local leaf containing $z$, denoted $T$, 
intersects $\gamma$
at a point $q$. 
If $f^{-1}(z) \in T$, 
then we find the fundamental domain in $T$. 
If $f^{-1}(z) \notin T$, then $f(q)\in K$, and this contradicts 
the fact that $z$ is the first returning point to $K$ along $W^s(p)$.

So, there is a fundamental domain of $W^s(p)$
contained in some local leaf in $K$, other than $L$.
Now we pick a lift $\widetilde p$ of $p$, and consider 
the lifted leaf $\widetilde{\mathcal F}(\widetilde p)$
containing 
$W^s(\widetilde p)$
and $W^u(\widetilde p)$. 
By previous paragraphs, 
there is some integer 
$(a,b)$ such that, 
%$\widetilde {\mathcal F}(\widetilde p)+(a,b)$ 
%intersect the lifted small region 
%$\widetilde K$. Moreover, 
the curves $\widetilde {\mathcal F}(\widetilde p)$ and 
$\widetilde {\mathcal F}(\widetilde p)+(a,b)$ bound an 
infinite strip $\widetilde H$, whose union with these two curves 
covers $\Bbb T^2$. 
Moreover,
we can find a small fundamental domain for 
$\widetilde f$ restricted to $\widetilde H$, namely 
$\widetilde K'\subset \widetilde K$, 
such that for any point 
$\widetilde z \in \widetilde H$ whose orbit is positively and negatively unbounded, 
its orbit must intersect $\widetilde K'$ (See Figure~\ref{fig7b}).

\begin{figure}[!ht]
\begin{center}
\includegraphics[width=8cm]{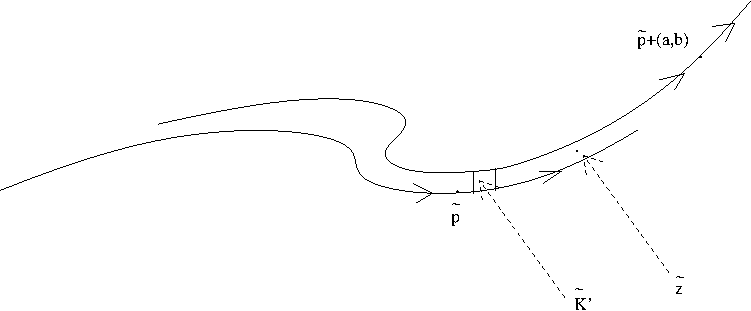}
\caption{A Fundamental Domain that All Unbounded Trajectories cross.}
\label{fig7b}
\end{center}
\end{figure}

On the other hand, we can choose $K \subset B$, 
where the disk $B$
was obtained at the beginning of the proof. 
Therefore, $K'=\pi(\widetilde K') \subset K\subset  B$
intersects the orbit of any chosen point in  $\mathcal S_{(\alpha,\beta)}$ (one
whose orbit is unbounded both in the future and past). And this is a
contradiction with the fact that $\mathcal S_{(\alpha,\beta)}$ avoids $B$.
\end{proof}

\bibliographystyle{plain}
%\addcontentsline{toc}{chapter}{Bibliography}
\bibliography{Perturbation}

\end{document}